\documentclass[11pt]{amsart}
\usepackage{latexsym}
\usepackage{amssymb}
\usepackage{amscd}
\usepackage{float}
\usepackage{graphicx}
\usepackage{amsfonts}
\usepackage{pb-diagram}
 \usepackage{amsmath,amscd}
\usepackage{setspace}

\newtheorem{thm}{Theorem}

\newtheorem{lemma}{Lemma}
\newtheorem{prop}{Proposition}
\newtheorem{defn}{Definition}
\newtheorem{remark}{Remark}

\newtheorem{ex}{Example}
\newtheorem{nt}{Notation}

\begin{document}

\title[On the $\mathcal{S}\left(L(p,1)\right)$ via braids]
  {The braid approach to the HOMFLYPT skein module of the lens spaces $L(p,1)$}

\author{Ioannis Diamantis}
\address{ International College Beijing,
China Agricultural University,
No.17 Qinghua East Road, Haidian District,
Beijing, {100083}, P. R. China.}
\email{ioannis.diamantis@hotmail.com}

\author{Sofia Lambropoulou}
\address{ Departament of Mathematics,
National Technical University of Athens,
Zografou campus,
{GR-15780} Athens, Greece.}
\email{sofia@math.ntua.gr}
\urladdr{http://www.math.ntua.gr/~sofia}

\keywords{HOMFLYPT skein module, solid torus, Iwahori--Hecke algebra of type B, mixed links, mixed braids, lens spaces. }

\subjclass[2010]{57M27, 57M25, 57Q45, 20F36, 20C08}

\thanks{This research  has been co-financed by the European Union (European Social Fund - ESF) and Greek national funds through the Operational Program ``Education and Lifelong Learning" of the National Strategic Reference Framework (NSRF) - Research Funding Program: THALES: Reinforcement of the interdisciplinary and/or inter-institutional research and innovation. }

\setcounter{section}{-1}

\date{}

\begin{abstract}
In this paper we present recent results toward the computation of the HOMFLYPT skein module of the lens spaces $L(p,1)$, $\mathcal{S}\left(L(p,1) \right)$, via braids. Our starting point is the knot theory of the solid torus ST and the Lambropoulou invariant, $X$, for knots and links in ST, the universal analogue of the HOMFLYPT polynomial in ST. The relation between $\mathcal{S}\left(L(p,1) \right)$ and $\mathcal{S}({\rm ST})$ is established in \cite{DLP} and it is shown that in order to compute $\mathcal{S}\left(L(p,1) \right)$, it suffices to solve an infinite system of equations obtained by performing all possible braid band moves on elements in the basis of $\mathcal{S}({\rm ST})$, $\Lambda$, presented in \cite{DL2}. The solution of this infinite system of equations is very technical and is the subject of a sequel paper \cite{DL3}.
\end{abstract}

\maketitle

\section{Introduction}\label{intro}

Skein modules were independently introduced in 1987 by Przytycki \cite{P} and Turaev \cite{Tu}. They generalize knot polynomials in $S^3$ to knot polynomials in arbitrary 3-manifolds. The essence is that skein modules are quotients of free modules over ambient isotopy classes of links in 3-manifolds by properly chosen local (skein) relations.

\smallbreak

Let $M$ be an oriented $3$-manifold, $R=\mathbb{Z}[u^{\pm1},z^{\pm1}]$, $\mathcal{L}$ the set of all oriented links in $M$ up to ambient isotopy in $M$ and let $S$ be the submodule of $R\mathcal{L}$ generated by the skein expressions $u^{-1}L_{+}-uL_{-}-zL_{0}$, where
$L_{+}$, $L_{-}$ and $L_{0}$ comprise a Conway triple represented schematically by the illustrations in Figure~\ref{skein}.

\begin{figure}[!ht]
\begin{center}
\includegraphics[width=1.7in]{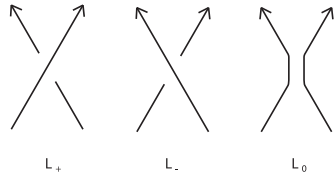}
\end{center}
\caption{The links $L_{+}, L_{-}, L_{0}$ locally.}
\label{skein}
\end{figure}

\noindent For convenience we allow the empty knot, $\emptyset$, and add the relation $u^{-1} \emptyset -u\emptyset =zT_{1}$, where $T_{1}$ denotes the trivial
knot. Then the {\it HOMFLYPT skein module} of $M$ is defined to be:

\begin{equation*}
\mathcal{S} \left(M\right)=\mathcal{S} \left(M;{\mathbb Z}\left[u^{\pm 1} ,z^{\pm 1} \right],u^{-1} L_{+} -uL_{-} -zL{}_{0} \right)={\raise0.7ex\hbox{$
R\mathcal{L} $}\!\mathord{\left/ {\vphantom {R\mathcal{L} S }} \right. \kern-\nulldelimiterspace}\!\lower0.7ex\hbox{$ S  $}}.
\end{equation*}

\smallbreak

Skein modules of $3$-manifolds have become very important algebraic tools in the study of $3$-manifolds, since their properties renders topological information about the $3$-manifolds. Unlike the Kauffman bracket skein module, the HOMFLYPT skein module of a $3$-manifold, also known as \textit{Conway skein module} and as \textit{third skein module}, is very hard to compute and very little is known so far. More precisely, $\mathcal{S}(S^3)=\mathbb{Z}[v^{\pm 1},z^{\pm 1}]$, where the empty link is a generator of the module (\cite{FYHLMO}, \cite{PT}). Also, for the solid torus ST, $\mathcal{S}({\rm ST})$ is a free, infinitely generated $\mathbb{Z}[u^{\pm1},z^{\pm1}]$-module isomorphic to the symmetric tensor algebra $SR\widehat{\pi}^0$, where $\widehat{\pi}^0$ denotes the conjugacy classes of non trivial elements of $\pi_1(\rm ST)$ (\cite{HK}, \cite{Tu}, \cite{La4}, \cite{DL2}). Further, let $F$ denote a surface. Then $\mathcal{S}(F \times I)$ is an algebra which, as an $R$ module, is a free module isomorphic to the symmetric tensor algebra, $SR\pi^o$, where $\pi^o$ denotes the conjugacy classes of nontrivial elements of $\pi_1(F)$ (\cite{P2}). Moreover, $\mathcal{S}(\mathbb{R}P^2 \hat{\times} I)$ is freely generated by standard oriented unlinks as presented in Figure~\ref{mroc} (\cite{Mro}) and $\mathcal{S}(S^1\times S^2)$ is freely generated by the empty link over a properly chosen ring (\cite{GZ1}). Finally, $\mathcal{S}(M_1 \# M_2)$ is isomorphic to $\mathcal{S}(M_1)\otimes \mathcal{S}(M_2)$ modulo torsion, where $M_1, M_2$ are oriented 3-manifolds and $M_1 \# M_2$ their connected sum (\cite{GZ2}).

\bigbreak

In \cite{GM} the HOMFLYPT skein module of the lens spaces $L(p,1)$ is computed using diagrammatic method. The diagrammatic method could in theory be generalized to the case of $L(p, q), q > 1$, but the diagrams become even more complex to analyze and several induction arguments
fail.

\begin{figure}
\begin{center}
\includegraphics[width=1.3in]{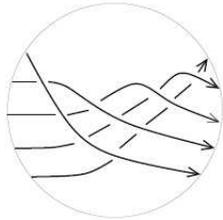}
\end{center}
\caption{An element in the basis of $\mathcal{S}(\mathbb{R}P^2 \hat{\times} I)$.}
\label{mroc}
\end{figure}

 In \cite{La4} the most generic analogue of the HOMFLYPT polynomial, $X$, for links in the solid torus $\rm ST$ has been derived from the generalized Hecke algebras of type $\rm B$, $\textrm{H}_{1,n}$, via a unique Markov trace constructed on them. This algebra was defined by Lambropoulou in \cite{La4} and is related to the knot theory of the solid torus, the Artin group of Coxeter group of type B, $B_{1, n}$, and to the affine Hecke algebra of type A. The Lambropoulou invariant $X$ recovers the HOMFLYPT skein module of ST, $\mathcal{S}({\rm ST})$, and is appropriate for extending the results to the lens spaces $L(p,q)$, since the combinatorial setting is the same as for $\rm ST$, only the braid equivalence includes the braid band moves (shorthanded to {\it bbm}), which reflect the surgery description of $L(p,q)$. For the case of $L(p,1)$, in order to extend $X$ to an invariant of links in $L(p,1)$ we need to solve an infinite system of equations resulting from the braid band moves. Namely we force:

\begin{equation}\label{eqbbm}
X_{\widehat{\alpha}} =  X_{\widehat{bbm(\alpha)}},
\end{equation}

\noindent for all $\alpha \in \bigcup_{\infty}B_{1,n}$ and for all possible slidings of $\alpha$.

\bigbreak

The above equations have particularly simple formulations with the use of a new basis, $\Lambda$, for the HOMFLYPT skein module of $\rm ST$, that we give in \cite{D, DL2}. This basis was predicted by Przytycki and is crucial in this paper, since {\it bbm}'s are naturally described by elements in this basis.

\smallbreak

In order to show that the set $\Lambda$ is a basis for $\mathcal{S}(\rm ST)$, we started in \cite{DL2} with the well-known basis of $\mathcal{S}(\rm ST)$, $\Lambda^{\prime}$, discovered independently in \cite{Tu} and \cite{HK} with diagrammatic methods, and a basis $\Sigma_n$ of the algebra $\textrm{H}_{1,n}$ and we followed the steps below:

\smallbreak

\begin{itemize}
\item[$\bullet$] An ordering relation in $\Lambda^{\prime}$ is defined and it is shown that the set is totally ordered. 

\smallbreak

\item[$\bullet$] Elements in $\Lambda^{\prime}$ are converted to linear combinations of elements in the new set $\Lambda$ as follows:

\smallbreak

\item[$\bullet$] Elements in $\Lambda^{\prime}$ are first converted to elements in the linear basis of ${\rm H}_{1, n}(q)$, $\Sigma_n$.
 
\smallbreak

\item[$\bullet$] Using conjugation, the gaps appearing in the indices of the looping generators in the monomials in $\bigcup_n\Sigma_n$ are managed.

\smallbreak

\item[$\bullet$] Using conjugation, the exponents of the looping generators are ordered.

\smallbreak

\item[$\bullet$] Using conjugation and stabilization moves, the `braiding tails' are removed from the above monomials and thus, the initial elements in $\bigcup_n\Sigma_n$ are converted to linear combination of elements in $\Lambda$.  

\smallbreak

\item[$\bullet$] Finally, the sets $\Lambda^{\prime}$ and $\Lambda$ are related via a block diagonal matrix, where each block is an infinite lower triangular matrix.

\smallbreak

\item[$\bullet$] The diagonal elements in the above matrix are invertible, making the matrix invertible and thus, the set $\Lambda$ is a basis for $\mathcal{S}(\rm ST)$.
\end{itemize}

\smallbreak

The new basis is appropriate for computing the HOMFLYPT skein module of the lens spaces $L(p,q)$ in general. Note that $\mathcal{S}(\rm ST)$ plays an important role in the study of HOMFLYPT skein modules of arbitrary c.c.o. $3$-manifolds, since every c.c.o. $3$-manifold can be obtained by surgery along a framed link in $S^3$ with unknotted components. The family of the lens spaces, $L(p,q)$, comprises the simplest example, since they are obtained by rational surgery on the unknot.

\smallbreak

Equations~(\ref{eqbbm}) are very controlled in the algebraic setting, because, as shown in \cite{DLP}, they can be performed only on elements in $\Lambda$. This is shown by following the technique developed in \cite{DL2}. The difference lies in the fact that here we deal with elements in $\Lambda$ and at the same time with their result after the performance of a {\it bbm} and we keep track of how {\it bbm}'s affect the steps described above. More precisely, in \cite{DLP} we followed the steps below:

\smallbreak

\begin{itemize}
\item[$\bullet$] Equations~(\ref{eqbbm}) boil down by linearity to considering only words in the canonical basis $\Sigma^{\prime}_{n}$ of the algebra ${\rm H}_{1,n}(q)$. 

\smallbreak

\item[$\bullet$] For words in $\bigcup_n\Sigma^{\prime}_{n}$ equations of the form $X_{\widehat{\alpha^{\prime}}}=X_{\widehat{bbm_{\pm 1}(\alpha^{\prime})}}$ are obtained, where $\widehat{bbm_{\pm 1}(\alpha^{\prime})}$ is the result of the performance of a braid band move on the {\it first} moving strand of the closed braid $\widehat{\alpha^{\prime}}$, and $\alpha^{\prime} \in \bigcup_n\Sigma^{\prime}_{n}$. 

\smallbreak

\item[$\bullet$] Then, elements in $\bigcup_n\Sigma^{\prime}_{n}$ are expressed to elements in the linear basis $\Sigma_n$ of ${\rm H}_{1,n}(q)$ and it is shown that the equations for words in $\bigcup_n\Sigma^{\prime}_{n}$ are equivalent to equations of the form $X_{\widehat{\alpha}}=X_{\widehat{bbm_{\pm 1}(\alpha)}}$, where $\alpha \in \bigcup_n\Sigma_n$. 

\smallbreak

\item[$\bullet$] A set $\Lambda^{aug}$ is then introduced, consisting of monomials in the looping generators $t_i$'s with no gaps  in the indices (as in $\Lambda$) but not necessarily ordered exponents.

\smallbreak

\item[$\bullet$] Equations for words in $\bigcup_n\Sigma_n$ are now reduced to equations obtained from elements in the ${\rm H}_{1,n}(q)$-module $\Lambda^{aug}$, where the braid band moves are performed on {\it any moving strand}. 

\smallbreak

\item[$\bullet$] Equations of the form $X_{\widehat{\beta}}=X_{\widehat{bbm_{\pm i}(\beta)}}$ are now obtained, where $\widehat{bbm_{\pm i}(\beta)}$ is the result of the performance of a braid band move on the $i^{th}$ moving strand of the closed braid $\widehat{\beta}$, and $\beta$ an element in the augmented set $\Lambda^{aug}$ followed by a `braiding tail'.

\smallbreak

\item[$\bullet$] Using conjugation, the exponents then become in decreasing order and equations obtained from elements in the ${\rm H}_{1,n}(q)$-module $\Lambda^{aug}$ by performing {\it bbm}'s on all moving strands are reduced to equations for words in ${\rm H}_{1,n}(q)$-module $\Lambda$ by performing {\it bbm}'s on all moving strands.

\smallbreak

\item[$\bullet$] The `braiding tails' from elements in the ${\rm H}_{1,n}(q)$-module $\Lambda$ are now eliminated and it is shown that equations for words in the ${\rm H}_{1,n}(q)$-module $\Lambda$ by performing braid band moves on any strand, are now reduced to equations obtained from elements in the basis $\Lambda$ of $\mathcal{S}({\rm ST})$ by performing braid band moves on every moving strand:
\end{itemize}

\begin{equation}\label{sysdlp}
\mathcal{S}\left( L(p,1) \right)\ =\ \frac{\mathcal{S}({\rm ST})}{<a-bbm_i(a)>},\ {\rm for\ all}\ i\ {\rm and\ for\ all}\ a\in \Lambda.
\end{equation}

\bigbreak

In \cite{DL3} we elaborate on the infinite system.

\bigbreak

The importance of our approach is that it can shed light on the problem of computing skein modules of arbitrary c.c.o. $3$-manifolds, since any $3$-manifold can be obtained by surgery on $S^3$ along unknotted closed curves. Indeed, one can use our results in order to apply a braid approach to the skein module of an arbitrary c.c.o. $3$-manifold. The main difficulty of the problem lies in selecting from the infinitum of band moves (or handle slide moves) some basic ones, solving the infinite system of equations and proving that there are no dependencies in the solutions. Note that the computation of $\mathcal{S}\left(L(p,1) \right)$ is equivalent to constructing all possible analogues of the HOMFLYPT or 2-variable Jones polynomial for knots and links in $L(p,1)$, since the linear dimension of $\mathcal{S}\left(L(p,1) \right)$ means the number of independent HOMFLYPT-type invariants defined on knots and links in $L(p,1)$.

\bigbreak

The paper is organized as follows: In \S\ref{basics} we recall the setting and the essential techniques and results from \cite{LR1,La1,LR2, DL1}. More precisely, we describe braid equivalence for knots and links in $L(p,1)$ and we present a sharpened version of the Reidemeister theorem for links in $L(p,1)$. We also provide geometric formulations of the braid equivalence via mixed braids in $S^3$ using the $L$-moves and the braid band moves and give algebraic formulations in terms of the mixed braid groups $B_{1,n}$. In \S\ref{SolidTorus} we present results from \cite{La4} and \cite{DL2}. More precisely, we recover the HOMFLYPT skein module of the solid torus ST, $\mathcal{S}({\rm ST})$, via algebraic techniques and we present a new basis for $\mathcal{S}({\rm ST})$, $\Lambda$, from which the braid band moves are naturally described. The aim of this section is to set a homogeneous ground in computing skein modules of c.c.o. $3$-manifolds in general via algebraic means. In \S~3 we derive the relation between $\mathcal{S}({\rm ST})$ and $\mathcal{S}\left(L(p,1) \right)$ and show that in order to compute $\mathcal{S}\left(L(p,1) \right)$ we only need to consider elements in the basis $\Lambda$ of $\mathcal{S}({\rm ST})$ and impose on the Lambropoulou invariant $X$ relations coming by performing all possible braid band moves on elements in $\Lambda$.

\section{Topological and algebraic tools}\label{basics}

\subsection{Mixed links and isotopy in $L(p,1)$}

We consider ST to be the complement of a solid torus in $S^3$. Then, an oriented link $L$ in ST can be represented by an oriented \textit{mixed link} in $S^{3}$, that is, a link in $S^{3}$ consisting of the unknotted {\it fixed part} $\widehat{I}$ representing the complementary solid torus in $S^3$ and the {\it moving part} $L$ that links with $\widehat{I}$. A \textit{mixed link diagram} is a diagram $\widehat{I}\cup \widetilde{L}$ of $\widehat{I}\cup L$ on the plane of $\widehat{I}$, where this plane is equipped with the top-to-bottom direction of $I$ (see \cite{LR1, La3}).

\begin{figure}
\begin{center}
\includegraphics[width=1.1in]{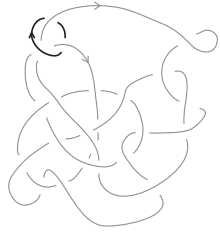}
\end{center}
\caption{A mixed link in $S^3$.}
\label{mlink}
\end{figure}

\smallbreak

It is well--known that the lens spaces $L(p,1)$ can be obtained from $S^3$ by surgery on the unknot with surgery coefficient $p$. Surgery along the unknot can be realized by considering first the complementary solid torus and then attaching to it a solid torus according to some homeomorphism on the boundary. Thus, isotopy in $L(p,1)$ can be viewed as isotopy in ST together with the {\it band moves} in $S^3$, which reflect the surgery description of the manifold, see Figure~\ref{bmov} (see \cite{LR1}). In \cite{DL1} we show that in order to describe isotopy for knots and links in a c.c.o. $3$-manifold, it suffices to consider only the type $\alpha$ band moves (see Figure~\ref{unknsurg}) and thus, isotopy between oriented links in $L(p,1)$ is reflected in $S^3$ by means of the following result (cf. \cite[Theorem~5.8]{LR1}, \cite[Theorem~6]{DL1}):

\smallbreak

{\it
Two oriented links in $L(p,1)$ are isotopic if and only if two corresponding mixed link diagrams of theirs differ by isotopy in {\rm ST} together with a finite sequence of the type $\alpha$ band moves.
}

\begin{figure}
\begin{center}
\includegraphics[width=4.5in]{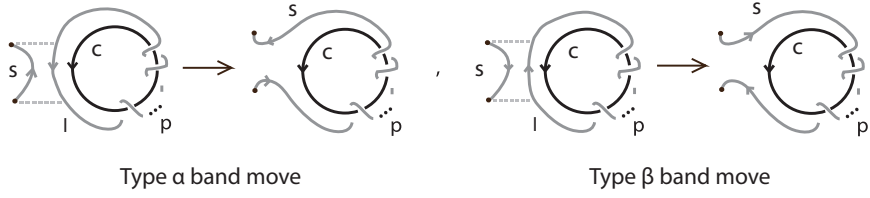}
\end{center}
\caption{The two types of band moves.}
\label{bmov}
\end{figure}

\begin{figure}
\begin{center}
\includegraphics[width=5in]{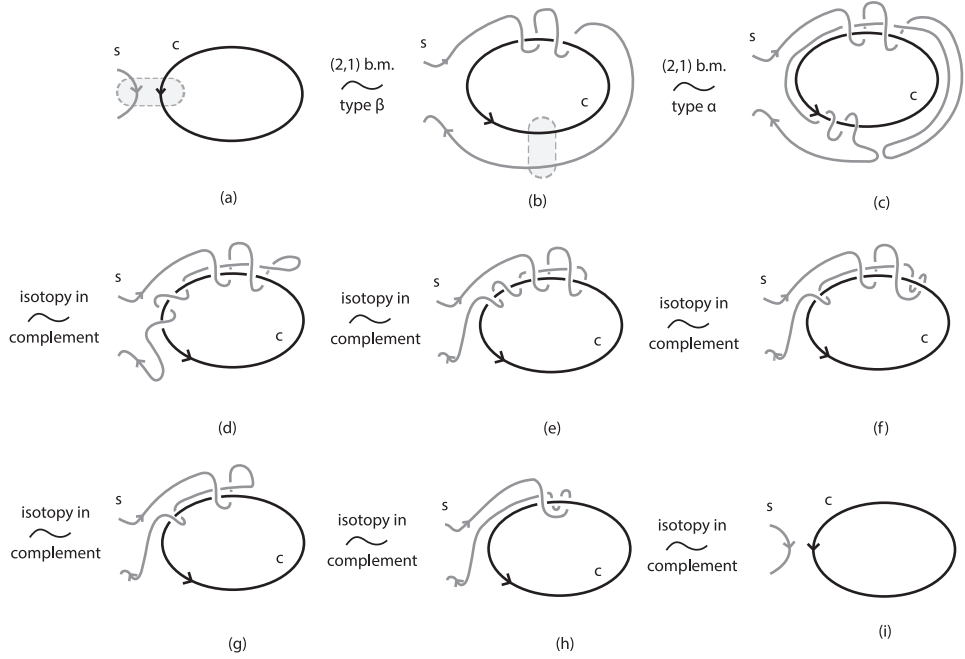}
\end{center}
\caption{A type-$\beta$ band move follows from a type-$\alpha$ band move in the case of integral surgery coefficient. }
\label{unknsurg}
\end{figure}

\subsection{Mixed braids for knots and links in $L(p,1)$}

By the Alexander theorem for knots in solid torus (\cite[Theorem~1]{La3}), a mixed link diagram $\widehat{I}\cup \widetilde{L}$ of $\widehat{I}\cup L$ may be turned into a \textit{mixed braid} $I\cup \beta $ with isotopic closure. This is a braid in $S^{3}$ where, without loss of generality, its first strand represents $\widehat{I}$, the fixed part, and the other strands, $\beta$, represent the moving part $L$. The subbraid $\beta$ is called the \textit{moving part} of $I\cup \beta $ (see Fig.~\ref{mbtoml}).

\begin{figure}
\begin{center}
\includegraphics[width=2in]{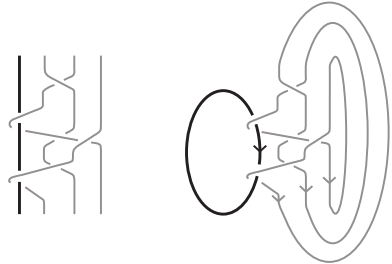}
\end{center}
\caption{The closure of a mixed braid to a mixed link.}
\label{mbtoml}
\end{figure}

The sets of mixed braids related to ST form groups, which are in fact the Artin braid groups of type B. The mixed braid group on $n$ moving strands is denoted as $B_{1,n}$ and it admits the following presentation:

\begin{equation}\label{presB}
B_{1,n} = \left< \begin{array}{ll}  \begin{array}{l} t, \sigma_{1}, \ldots ,\sigma_{n-1}  \\ \end{array} & \left| \begin{array}{l}
\sigma_{1}t\sigma_{1}t=t\sigma_{1}t\sigma_{1} \ \   \\
 t\sigma_{i}=\sigma_{i}t, \quad{i>1}  \\
{\sigma_i}\sigma_{i+1}{\sigma_i}=\sigma_{i+1}{\sigma_i}\sigma_{i+1}, \quad{ 1 \leq i \leq n-2}   \\
 {\sigma_i}{\sigma_j}={\sigma_j}{\sigma_i}, \quad{|i-j|>1}  \\
\end{array} \right.  \end{array} \right>,
\end{equation}

\noindent where the `braiding' generators $\sigma _{i}$ and the `looping' generator $t$ are illustrated in Figure~\ref{gen} (\cite{La3}).

\begin{figure}
\begin{center}
\includegraphics[width=2.3in]{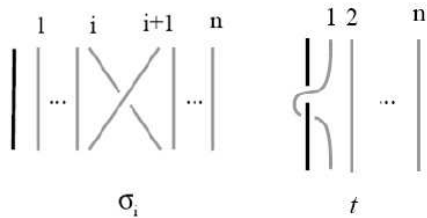}
\end{center}
\caption{The generators of $B_{1,n}$.}
\label{gen}
\end{figure}

\smallbreak

In $B_{1,n}$ we also define the elements:
\begin{equation}\label{lgen}
t^{\prime}_{i}:=\sigma_{i}\sigma_{i-1}\ldots \sigma_{1}t\sigma_{1}^{-1}\ldots \sigma_{i-1}^{-1}\sigma_{i}^{-1}\ {\rm and}\ t_{i}:=\sigma_{i}\sigma_{i-1}\ldots \sigma_{1}t\sigma_{1} \ldots \sigma_{i-1}\sigma_{i},
\end{equation}
for $i= 0, 1, \ldots, n-1$, where $t_0=t=t_0^{\prime}$. For illustrations see Figure~\ref{genh}.

\begin{remark}\label{tprt}\rm
In \cite[Prop.~1]{La4} it is shown using the Artin combing, that every element in $B_{1, n}$ can be written in the form $\tau^{\prime}\cdot w$, where $\tau^{\prime}$ is a word in the looping elements $t_i^{\prime}$ and $w\in B_n$ is the `braiding tail'. Furthermore, the looping elements $t_i^{\prime}$ are conjugates. On the other hand, the elements $t_i$ commute. Moreover, the $t_i$'s are consistent with the braid band move used in the link isotopy in $L(p,1)$, in the sense that a braid band move can be described naturally with the use of the $t_i$'s.  
\end{remark}

\begin{figure}\label{loopttpr}
\begin{center}
\includegraphics[width=3.2in]{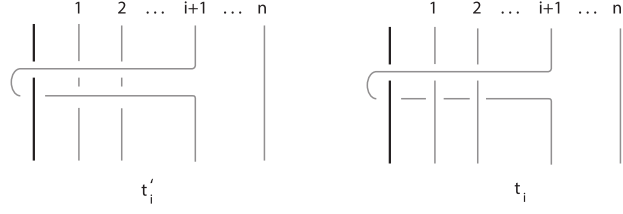}
\end{center}
\caption{The elements $t^{\prime}_{i}$ and $t_{i}$.}
\label{genh}
\end{figure}

\subsection{Braid equivalence for knots and links in $L(p,1)$}

In \cite{LR1} the authors give a sharpened version of the classical Markov theorem for braid equivalence in $S^3$ based only on one type of moves, the $L$-moves: An {\it $L$-move} on a mixed braid $I \bigcup \beta$, consists in cutting an arc of the moving sub-braid open and pulling the upper cut point downward and
the lower upward, so as to create a new pair of braid strands with corresponding endpoints, and such that both strands cross entirely {\it over} or {\it under} with the rest of the braid. Stretching the new strands over will give rise to an {\it $L_o$-move\/} and under to an {\it $L_u$-move\/}. For an illustration see Figure~\ref{lmove}. An algebraic $L$-move has the following algebraic expression for an $L_o$-move and an $L_u$-move respectively:

\begin{equation}\label{l-m}
\begin{array}{lll}
\alpha &=& \alpha_1\alpha_2 \stackrel{L_o}{\sim}
\sigma_i^{-1}\ldots \sigma_n^{-1} \alpha_1' \sigma_{i-1}^{-1}\ldots
\sigma_{n-1}^{-1}\sigma_n^{\pm 1}\sigma_{n-1} \ldots \sigma_i
\alpha_2' \sigma_n \ldots \sigma_i \\
\alpha&=&\alpha_1\alpha_2 \stackrel{L_u}{\sim}
\sigma_i\ldots \sigma_n \alpha_1' \sigma_{i-1}\ldots
\sigma_{n-1}\sigma_n^{\pm 1}\sigma_{n-1}^{-1}\ldots\sigma_i^{-1}
\alpha_2' \sigma_n^{-1}\ldots\sigma_i^{-1}
\end{array}
\end{equation}

\noindent  where $\alpha_1$,
 $\alpha_2$ are elements of $B_{1,n}$ and $\alpha_1'$,
 $\alpha_2' \in B_{1,n+1}$ are obtained from $\alpha_1$,
 $\alpha_2$ by replacing each $\sigma_j$ by $\sigma_{j+1}$ for
 $j=i,\ldots,n-1$.

\begin{figure}
\begin{center}
\includegraphics[width=4.9in]{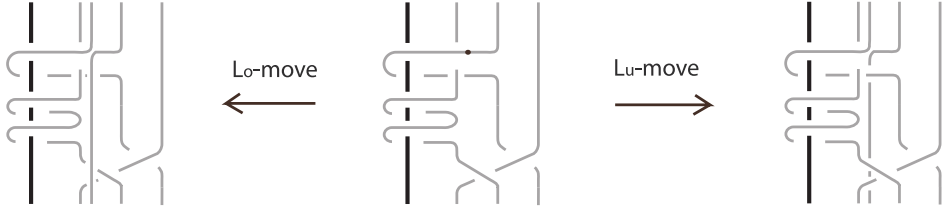}
\end{center}
\caption{A mixed braid and the two types of $L$-moves}
\label{lmove}
\end{figure}

In order to translate isotopy for links in $L(p,1)$ into mixed braid equivalence, we first define a {\it braid band move} to be a move between mixed braids, which is a type $\alpha$ band move between their closures. It starts with a little band oriented downward, which, before sliding along a surgery strand, gets one twist {\it positive\/} or {\it negative\/} (see Figure~\ref{bbm12}).

\begin{figure}
\begin{center}
\includegraphics[width=5in]{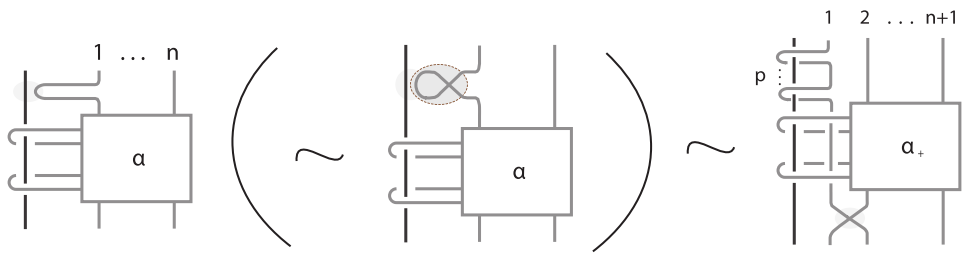}
\end{center}
\caption{Braid band move performed on the first moving strand.}
\label{bbm12}
\end{figure}

\smallbreak

Then, isotopy in $L(p,1)$ is translated on the level of mixed braids by means of the following theorem:

\begin{thm}{\cite[Theorem~5]{LR2}} \label{markov}
 Let $L_{1} ,L_{2}$ be two oriented links in $L(p,1)$ and let $I\cup \beta_{1} ,{\rm \; }I\cup \beta_{2}$ be two corresponding mixed braids in $S^{3}$. Then $L_{1}$ is isotopic to $L_{2}$ in $L(p,1)$ if and only if $I\cup \beta_{1}$ is equivalent to $I\cup \beta_{2}$ in $\mathop{\cup }\limits_{n=1}^{\infty } B_{1,n}$ by the following moves:
\[ \begin{array}{clll}
(i)  & Conjugation:         & \alpha \sim \beta^{-1} \alpha \beta, & \alpha ,\beta \in B_{1,n} \\
(ii) & Stabilization\ moves: &  \alpha \sim \alpha \sigma_{n}^{\pm 1} \in B_{1,n+1}, & \alpha \in B_{1,n} \\
(iii) & Loop\ conjugation: & \alpha \sim t^{\pm 1} \alpha t^{\mp 1}, & \alpha \in B_{1,n} \\
(iv) & Braid\ band\ moves: & \alpha \sim {t}^p \alpha_+\cdot \sigma_1^{\pm 1}, & \alpha_+\in B_{1, n+1}
\end{array} \]

\noindent where $\alpha_+$ is the word $\alpha$ with all indices shifted by 1 (see Figure~\ref{bbm12}). Equivalently, $L_{1}$ is isotopic to $L_{2}$ if and only if $I\cup \beta_{1}$ and $I\cup \beta_{2}$ differ by moves (iii) and (iv) as above, while moves (i) and (ii) are replaced by the two types of $L$-moves, see Eq.~(\ref{l-m}).
\end{thm}

\begin{nt}\rm
We denote a braid band move by {\it bbm} and, specifically, the result of a positive or negative braid band move performed on the $i^{th}$-moving strand of a mixed braid $\beta$ by $bbm_{\pm i}(\beta)$.
\end{nt}

Note that in the statement of Theorem~\ref{markov} in \cite{LR2} the braid band moves take place on the last strand of a mixed braid. Clearly, this is equivalent to performing the braid band moves on the first moving strand or, in fact, on any specified moving strand of the mixed braid. Indeed, in a mixed braid $\beta\in B_{1,n}$ consider the last moving strand of $\beta$ approaching the surgery strand $I$ from the right. Before performing a {\it bbm} we apply conjugation (isotopy in ST) and we obtain an equivalent mixed braid where the first strand is now approaching $I$ (see Figure~\ref{bbmconj}). In terms of diagrams we have the following (\cite[Lemma~1]{DLP}):

\[
\begin{array}{ccccc}
\beta & \sim & (\sigma_{i-1}\ldots \sigma_1 \sigma_1^{-1}\ldots \sigma_{i-1}^{-1})\cdot \beta & {\sim} & \underset{\alpha}{\underbrace{(\sigma_1^{-1}\ldots \sigma_{i-1}^{-1})\cdot \beta \cdot (\sigma_{i-1}\ldots \sigma_1)}} \\
\downarrow &  &  & & \downarrow \\
bbm_{\pm n}(\beta) &  &  & = & bbm_{\pm 1}(\alpha) 
\end{array}
\]

\begin{figure}
\begin{center}
\includegraphics[width=4.3in]{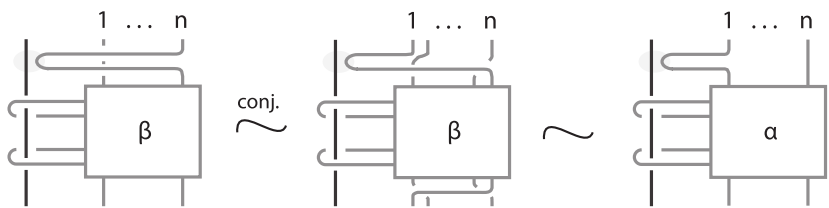}
\end{center}
\caption{ A {\it bbm} may always be assumed to be performed on the first moving strand of a mixed braid.}
\label{bbmconj}
\end{figure}

\subsection{The generalized Iwahori-Hecke algebra of type B}

It is well-known that $B_{1,n}$ is the Artin group of the Coxeter group of type B, which is related to the Hecke algebra of type B, $\textrm{H}_{n}{(q,Q)}$ and to the cyclotomic Hecke algebras of type B. In \cite{La4} it has been established that all these algebras are related to the knot theory of ST. The basic one is $\textrm{H}_{n}{(q,Q)}$, a presentation of which is obtained from the presentation (\ref{presB}) of the Artin group $B_{1,n}$ by corresponding the braiding generator $\sigma_i$ to $g_i$ and by adding the quadratic relations:

\begin{equation}\label{quad}
{g_{i}^2=(q-1)g_{i}+q}
\end{equation}

\noindent and the relation $t^{2} =\left(Q-1\right)t+Q$. The cyclotomic Hecke algebras of type B are denoted $\textrm{H}_{n}(q,d),\ d\in \mathbb{N}$, and they admit presentations that are obtained by the quadratic relation~(\ref{quad}) and the relation $t^d=(t-u_{1})(t-u_{2}) \ldots (t-u_{d})$. In \cite{La4} also the \textit{generalized Iwahori--Hecke algebra of type B}, $\textrm{H}_{1,n}(q)$, is introduced, as the quotient of ${\mathbb C}\left[q^{\pm 1} \right]B_{1,n}$ over the quadratic relations (\ref{quad}). Namely:

\begin{equation*}
\textrm{H}_{1,n}(q)= \frac{{\mathbb C}\left[q^{\pm 1} \right]B_{1,n}}{ \langle \sigma_i^2 -\left(q-1\right)\sigma_i-q \rangle}.
\end{equation*}

Note that in $\textrm{H}_{1,n}(q)$ the generator $t$ satisfies no polynomial relation, making the algebra $\textrm{H}_{1,n}(q)$ infinite dimensional, and as observed by T. tom Dieck, ${\rm H}_{1, n}$ is closely related to the affine Hecke algebra of type A, $\widetilde{\textrm{H}}_n(q)$. In \cite{La4} the algebra $\textrm{H}_{1,n}(q)$ is denoted as $\textrm{H}_{n}(q, \infty)$.

\smallbreak

In \cite{Jo} V.F.R. Jones gives the following linear basis for the Iwahori-Hecke algebra of type A, $\textrm{H}_{n}(q)$:

$$ S =\left\{(g_{i_{1} }g_{i_{1}-1}\ldots g_{i_{1}-k_{1}})(g_{i_{2} }g_{i_{2}-1 }\ldots g_{i_{2}-k_{2}})\ldots (g_{i_{p} }g_{i_{p}-1 }\ldots g_{i_{p}-k_{p}})\right\}, $$

\noindent for $1\le i_{1}<\ldots <i_{p} \le n-1{\rm \; }$.

\smallbreak

The basis $S$ yields directly an inductive basis for $\textrm{H}_{n}(q)$, which is used in the construction of the Ocneanu trace, leading to the HOMFLYPT or $2$-variable Jones polynomial.

\bigbreak

We also introduce in ${\rm H}_{1, n}(q)$ the `looping elements'
\begin{equation}
t_0^{\prime}\ =\ t_0\ :=\ t, \quad t_i^{\prime}\ =\ g_i\ldots g_1tg_1^{-1}\ldots g_i^{-1} \quad {\rm and}\quad t_i\ =\ g_i\ldots g_1tg_1\ldots g_i
\end{equation}

\noindent that correspond to the elements of Eq.~\ref{lgen} of $B_{1,n}$. 

\smallbreak

From \cite{La4} we have:

\begin{thm}{\cite[Proposition~1 \& Theorem~1]{La4}} \label{basesH}
The following sets form linear bases for ${\rm H}_{1,n}(q)$:
\[
\begin{array}{llll}
 (i) & \Sigma_{n} & = & \{t_{i_{1} } ^{k_{1} } \ldots t_{i_{r}}^{k_{r} } \cdot \sigma \} ,\ {\rm where}\ 0\le i_{1} <\ldots <i_{r} \le n-1,\\
 (ii) & \Sigma^{\prime} _{n} & = & \{ {t^{\prime}_{i_1}}^{k_{1}} \ldots {t^{\prime}_{i_r}}^{k_{r}} \cdot \sigma \} ,\ {\rm where}\ 0\le i_{1} < \ldots <i_{r} \le n-1, \\
\end{array}
\]
\noindent where $k_{1}, \ldots ,k_{r} \in {\mathbb Z}$ and $\sigma$ a basic element in ${\rm H}_{n}(q)$.
\end{thm}

For an illustration of the $t_i$'s and the $t_i^{\prime}$'s recall Figure~\ref{genh}.

\begin{remark}\label{conind}\rm
Note that the form of the elements in the set $\Sigma_n^{\prime}$ is consistent with the Artin combing in the mixed braid group $B_{1, n}$ (recall Remark~\ref{tprt}). However the looping part of a mixed braid after the combing contains elements of a free group on $n$ generators, so the indices as well as the exponents have no restrictions. In Theorem~\ref{basesH} though, the indices of the $t_i$'s and the $t^{\prime}_i$'s in the above sets are ordered but are not necessarily consecutive, neither do they need to start from $t$. Also, the exponents are arbitrary. 
\end{remark}

\begin{nt}\label{SSpr} \rm
We shall denote 
\begin{equation}
\Sigma\ :=\ \bigcup_n \Sigma_n \quad {\rm and\ similarly}\quad \Sigma^{\prime}\ :=\ \bigcup_n \Sigma_n^{\prime}.
\end{equation}
\end{nt}

\section{The HOMFLYPT skein module of the solid torus}\label{SolidTorus}

In \cite{Tu}, \cite{HK} the HOMFLYPT skein module of the solid torus ST has been computed using diagrammatic methods by means of the following theorem:

\begin{thm}[Turaev, Kidwell--Hoste] \label{turaev}
The skein module $\mathcal{S}({\rm ST})$ is a free, infinitely generated $\mathbb{Z}[u^{\pm1},z^{\pm1}]$-module isomorphic to the symmetric
tensor algebra $SR\widehat{\pi}^0$, where $\widehat{\pi}^0$ denotes the conjugacy classes of non trivial elements of $\pi_1(\rm ST)$.
\end{thm}

In the diagrammatic setting of \cite{Tu} and \cite{HK}, ST is considered as ${\rm Annulus} \times {\rm Interval}$. A basic element of $\mathcal{S}({\rm ST})$ in the context of \cite{Tu, HK}, is illustrated in Figure~\ref{tur}. The HOMFLYPT skein module of ST is particularly important, because any closed, connected, oriented (c.c.o.) $3$-manifold can be obtained by surgery along a framed link in $S^3$ with unknotted components. Such component can be viewed as a solid torus complementing ST in $S^3$.

\begin{figure}
\begin{center}
\includegraphics[width=1.3in]{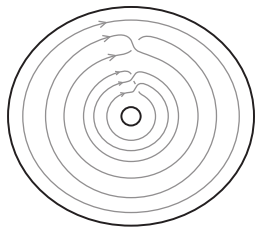}
\end{center}
\caption{A basic element of $\mathcal{S}({\rm ST})$.}
\label{tur}
\end{figure}

\subsection{Recovering $\mathcal{S}\left({\rm ST}\right)$ using algebraic means}

In \cite{La4} the bases $\Sigma^{\prime}_{n}$ are used for constructing a Markov trace on $\bigcup _{n=1}^{\infty }\textrm{H}_{1,n}(q)$.

\begin{thm}{\cite[Theorem~6]{La4}} \label{tr}
Given $z,s_{k}$, with $k\in {\mathbb Z}$ specified elements in $R={\mathbb Z}\left[q^{\pm 1} \right]$, there exists
a unique linear Markov trace function
\begin{equation*}
{\rm tr}:\bigcup _{n=1}^{\infty }{\rm H}_{1,n}(q)  \to R\left(z,s_{k} \right),k\in {\mathbb Z}
\end{equation*}

\noindent determined by the rules:
\[
\begin{array}{lllll}
(1) & {\rm tr}(ab) & = & {\rm tr}(ba) & \quad {\rm for}\ a,b \in {\rm H}_{1,n}(q) \\
(2) & {\rm tr}(1) & = & 1 & \quad {\rm for\ all}\ {\rm H}_{1,n}(q) \\
(3) & {\rm tr}(ag_{n}) & = & z{\rm tr}(a) & \quad {\rm for}\ a \in {\rm H}_{1,n}(q) \\
(4) & {\rm tr}(a{t^{\prime}_{n}}^{k}) & = & s_{k}{\rm tr}(a) & \quad {\rm for}\ a \in {\rm H}_{1,n}(q),\ k \in {\mathbb Z}. \\
\end{array}
\]
\end{thm}

Note that the use of the looping elements $t_i^{\prime}$ enable the trace ${\rm tr}$ to be defined by just extending  by rule (4) the three rules of the Ocneanu trace on the algebras ${\rm H}_n(q)$ \cite{Jo}, recall Remark~\ref{tprt}. Using $\textrm{tr}$ the second author constructed a universal HOMFLYPT-type invariant for oriented links in ST. Namely, let $\mathcal{L}$ denote the set of oriented links in ST. Then:

\begin{thm}{\cite[Definition~1]{La4}} \label{inv}
The function $X:\mathcal{L}$ $\rightarrow R(z,s_{k})$

\begin{equation*}
X_{\widehat{\alpha}} = \Delta^{n-1}
{\rm tr}\left(\pi \left(\alpha \right)\right),
\end{equation*}

\noindent where $\Delta:=-\frac{1-\lambda q}{\sqrt{\lambda } \left(1-q\right)} \left(\sqrt{\lambda } \right)^{e}$, $\lambda := \frac{z+1-q}{qz}$, $\alpha \in B_{1,n}$ is a word in the $\sigma _{i}$'s and $t^{\prime}_{i} $'s, $\widehat{\alpha}$ is the closure of $\alpha$, $e$ is the exponent sum of the $\sigma _{i}$'s in $\alpha $, and $\pi$ the canonical map of $B_{1,n}$ on ${\rm H}_{1,n}(q)$, such that $t\mapsto t$ and $\sigma _{i} \mapsto g_{i} $, is an invariant of oriented links in {\rm ST}.
\end{thm}

\bigbreak

In the braid setting of \cite{La4}, the elements of $\mathcal{S}({\rm ST})$ correspond bijectively to the elements of the following set $\Lambda^{\prime}$:

\begin{equation}\label{Lpr}
\Lambda^{\prime}=\{ {t^{k_0}}{t^{\prime}_1}^{k_1} \ldots
{t^{\prime}_n}^{k_n}, \ k_i \in \mathbb{Z}\setminus\{0\}, \ k_i \geq k_{i+1}\ \forall i,\ n\in \mathbb{N} \}.
\end{equation}

\noindent So, we have that $\Lambda^{\prime}$ is a basis of $\mathcal{S}({\rm ST})$ in terms of braids. Note that $\Lambda^{\prime}$ is a subset of $\bigcup_n{\textrm{H}_{1,n}}$ and, in particular, $\Lambda^{\prime}$ is a subset of $\Sigma^{\prime}$. Note also that in contrast to elements in $\Sigma^{\prime}$, the elements in $\Lambda^{\prime}$ have no gaps in the indices, the exponents are ordered and there are no `braiding tails'. 

\begin{remark}\rm
The Lambropoulou invariant $X$ recovers $\mathcal{S}({\rm ST})$, because it gives distinct values to distinct elements of $\Lambda^{\prime}$, since $tr(t^{k_0}{t^{\prime}_1}^{k_1} \ldots {t^{\prime}_n}^{k_n})=s_{k_n}\ldots s_{k_1}s_{k_0}$.
\end{remark}

\noindent Note also that the invariant $X$ is defined by the skein relation:
\[
\frac{1}{\sqrt{q}\sqrt{\lambda}}X_{L_+}-\sqrt{q}\sqrt{\lambda}X_{L_-}\ =\ \left(\sqrt{q}-\frac{1}{\sqrt{q}} \right) X_{L_o}
\]
\noindent and an infinitum of initial conditions, one for each element in $\mathcal{S}({\rm ST})$ as shown in Figure~12. Namely, if $U$ denotes the unknot, then $X_{U}=1$ and $X_{\widehat{t^{k_0}{t^{\prime}_1}^{k_1} \ldots {t^{\prime}_n}^{k_n}}}=\Delta^{n-1}\cdot s_{k_n}\ldots s_{k_1}s_{k_0}$.

\subsection{The new basis, $\Lambda$, of $\mathcal{S}({\rm ST})$}\label{lamb}

In \cite{DL2} we give a different basis $\Lambda$ for $\mathcal{S}({\rm ST})$, which was predicted by J. Przytycki and which is described in Eq.~\ref{basis} in open braid form.

\begin{thm}{\cite[Theorem~2]{DL2}}\label{newbasis}
The following set is a $\mathbb{Z}[q^{\pm1}, z^{\pm1}]$-basis for $\mathcal{S}({\rm ST})$:
\begin{equation}\label{basis}
\Lambda=\{t^{k_0}t_1^{k_1}\ldots t_n^{k_n},\ k_i \in \mathbb{Z}\setminus\{0\},\ k_i \geq k_{i+1}\ \forall i,\ n \in \mathbb{N} \}.
\end{equation}
\end{thm}

The importance of the new basis $\Lambda$ of $\mathcal{S}({\rm ST})$ lies in the simplicity of the algebraic expression of a braid band move, which extends the link isotopy in ST to link isotopy in $L(p,1)$ and this fact was our motivation for establishing this new basis $\Lambda$ (recall Theorem~1(iv), Remark~\ref{tprt} and Figure~\ref{bbm12}).

\begin{figure}
\begin{center}
\includegraphics[width=1.6in]{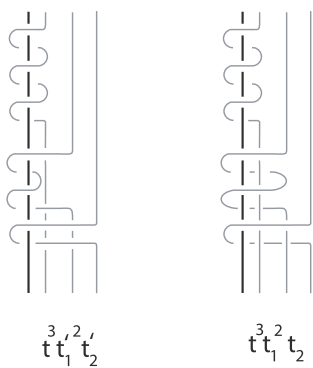}
\end{center}
\caption{Elements in two different bases of $\mathcal{S}({\rm ST})$.}
\label{basel}
\end{figure}

\begin{figure}
\begin{center}
\includegraphics[width=1.3in]{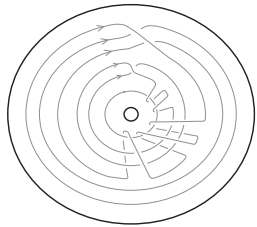}
\end{center}
\caption{An element of the new basis $\Lambda$.}
\label{tbasis}
\end{figure}

Notice that comparing the set $\Lambda$ with the set $\Sigma$, we observe that there are no gaps in the indices of the $t_i$'s and the exponents are in decreasing order. Also, there are no `braiding tails' in the words in $\Lambda$.

\smallbreak

Our method for proving Theorem~\ref{newbasis} is the following:

\smallbreak

$\bullet$ We define total orderings in the sets $\Lambda^{\prime}$ and $\Lambda$,

\smallbreak

$\bullet$ we show that the two ordered sets are related via a lower triangular infinite matrix with invertible elements on the diagonal, and

\smallbreak

$\bullet$ using this matrix, we show that the set $\Lambda$ is linearly independent.

\smallbreak

In order to relate the two sets via a lower triangular infinite matrix we start with elements in the basic set $\Lambda^{\prime}$ and we first convert them into sums of elements in $\Sigma$, containing the linear bases of the algebras ${\rm H}_{1, n}(q)$. These elements consist of two parts: arbitrary monomials in the $t_i$'s followed by `braiding tails' in the bases of the algebras ${\rm H}_n(q)$.
 Then, these elements are converted into elements in the set $\Lambda$ by:

\begin{itemize}
\item[$\bullet$] managing the gaps in the indices, 
\smallbreak
\item[$\bullet$] by ordering the exponents of the $t_i$'s and
\smallbreak
\item[$\bullet$] by eliminating the `braiding tails'. 
\end{itemize}

It is worth mentioning that these procedures are not independent in the sense that when, for example one manages the gaps in the indices of the looping generators $t_i$'s, `braiding tails' may occur and also the exponents of the $t_i$'s may alter. Similarly, when the `braiding tails' are eliminated, gaps in the indices of the $t_i$'s might occur. This is a long procedure that eventually stops and only elements in the set $\Lambda$ remain.

\subsection{An ordering for the sets $\Sigma^{\prime}, \Sigma, \Lambda^{\prime}$ and $\Lambda$}

For defining orderings in the sets $\Sigma, \Sigma^{\prime}, \Lambda$ and $\Lambda^{\prime}$ we need the notion of the {\it index} of a word $w$ in any of these sets, denoted $ind(w)$. In $\Lambda^{\prime}$ or $\Lambda$ $ind(w)$ is defined to be the highest index of the $t_i^{\prime}$'s, resp. of the $t_i$'s in $w$. Similarly, in $\Sigma^{\prime}$ or $\Sigma$, $ind(w)$ is defined as above by ignoring possible gaps in the indices of the looping generators and by ignoring the braiding parts in the algebras $\textrm{H}_{n}(q)$. Moreover, the index of a monomial in $\textrm{H}_{n}(q)$ is equal to $0$.

	\begin{ex} \rm
	\[
	\begin{array}{llllll}
{i.} & ind({t^{\prime}}^{k_0}{t^{\prime}_1}^{k_1}\ldots {{t^{\prime}}_n}^{k_n}) & = & n & = & ind(t^{k_0}{t_1}^{k_1}\ldots {t_n}^{k_n})\\
{ii.} & ind(t^{k_0}t_2^{k_2}\cdot \sigma)                                      & = & 2 & &,\ \sigma \in {\rm H}_n(q)\\
{iii.} & ind(\sigma)                                                           & = & 0 & &,\  \sigma \in {\rm H}_n(q)
	\end{array}
	\]
	\end{ex}

\smallbreak

We now proceed with defining an ordering relation in the sets $\Sigma$ and $\Sigma^{\prime}$, which passes to their respective subsets $\Lambda$ and $\Lambda^{\prime}$:

\begin{defn}{\cite[Definition~2]{DL2}} \label{order}
\rm
Let $w={t^{\prime}_{i_1}}^{k_1}\ldots {t^{\prime}_{i_{\mu}}}^{k_{\mu}}\cdot \beta_1$ and $u={t^{\prime}_{j_1}}^{\lambda_1}\ldots {t^{\prime}_{j_{\nu}}}^{\lambda_{\nu}}\cdot \beta_2$ in $\Sigma^{\prime}$, where $k_t , \lambda_s \in \mathbb{Z}$ for all $t,s$ and $\beta_1, \beta_2 \in H_n(q)$. Then, we define the following ordering in $\Sigma^{\prime}$:

\smallbreak

\begin{itemize}
\item[(a)] If $\sum_{i=0}^{\mu}k_i < \sum_{i=0}^{\nu}\lambda_i$, then $w<u$.

\vspace{.1in}

\item[(b)] If $\sum_{i=0}^{\mu}k_i = \sum_{i=0}^{\nu}\lambda_i$, then:

\vspace{.1in}

\noindent  (i) if $ind(w)<ind(u)$, then $w<u$,

\vspace{.1in}

\noindent  (ii) if $ind(w)=ind(u)$, then:

\vspace{.1in}

\noindent \ \ \ \ ($\alpha$) if $i_1=j_1, \ldots , i_{s-1}=j_{s-1}, i_{s}<j_{s}$, then $w>u$,

\vspace{.1in}

\noindent \ \ \  ($\beta$) if $i_t=j_t$ for all $t$ and $k_{\mu}=\lambda_{\mu}, k_{\mu-1}=\lambda_{\mu-1}, \ldots, k_{i+1}=\lambda_{i+1}, |k_i|<|\lambda_i|$, then $w<u$,

\vspace{.1in}

\noindent \ \ \  ($\gamma$) if $i_t=j_t$ for all $t$ and $k_{\mu}=\lambda_{\mu}, k_{\mu-1}=\lambda_{\mu-1}, \ldots, k_{i+1}=\lambda_{i+1}, |k_i|=|\lambda_i|$ and $k_i>\lambda_i$, then $w<u$,

\vspace{.1in}

\noindent \ \ \ \ ($\delta$) if $i_t=j_t\ \forall t$ and $k_i=\lambda_i$, $\forall i$, then $w=u$.

\end{itemize}

The ordering in the set $\Sigma$ is defined as in $\Sigma^{\prime}$, where $t_i^{\prime}$'s are replaced by $t_i$'s.
\end{defn}

\begin{ex} \rm
\[
\begin{array}{rlllll}
{\rm i.} & {t}^2{t_1^{\prime}}^8 & < &  {t}^6{t_1^{\prime}}^{10} & {\rm since}\ 2+8\ <\ 6+10\ &,\ {\rm \left(Def.~1a\right)}\\\
{\rm ii.} & {t}^2{t_1^{\prime}}^8 & < &  {t}^3{t_1^{\prime}}^4{t_2^{\prime}}^3 & {\rm since}\ ind({t}^2{t_1^{\prime}}^8)<ind({t}^3{t_1^{\prime}}^4{t_2^{\prime}}^3) &,\ {\rm \left(Def.~1b(i)\right)}\\
{\rm iii.} & {t}^2{t_1^{\prime}}^3{t_3^{\prime}}^4{t_4^{\prime}}^8 & < &  {t}^8{t_1^{\prime}}{t_2^{\prime}}{t_4^{\prime}}^7 & {\rm since}\ 3\ =\ i_3\ >\ j_3\ =\ 2 &,\ {\rm \left(Def.~1b(ii)(\alpha)\right)}\\
{\rm iv.} & {t}^2{t_1^{\prime}}^4{t_3^{\prime}}{t_4^{\prime}}^3 & < &   {t}^{14}{t_1^{\prime}}^{-8}{t_3^{\prime}}{t_4^{\prime}}^3 & {\rm since}\ |4|\ <\ |-8| &,\ {\rm \left(Def.~1b(ii)(\beta)\right)}\\
{\rm v.} & t{t_1^{\prime}}^4{t_3^{\prime}}{t_4^{\prime}}^3 & < &  {t}^{10}{t_1^{\prime}}^{-4}{t_3^{\prime}}{t_4^{\prime}}^3 & {\rm since}\ -4\ <\ 4 &,\ {\rm \left(Def.~1b(ii)(\gamma)\right)}
\end{array}
\]
\end{ex}

In order to eventually get to the infinite matrix relating the two basic sets $\Lambda^{\prime}$ and $\Lambda$ we need to define the \textit{subsets of level $k$}, $\Lambda_{(k)}$ and $\Lambda^{\prime}_{(k)}$, of $\Lambda$ and $\Lambda^{\prime}$ respectively ({\cite[Definition~3]{DL2}}), to be the sets:

\begin{equation}
\begin{array}{l}
\Lambda_{(k)}:=\{t_0^{k_0}t_1^{k_1}\ldots t_{m}^{k_m} | \sum_{i=0}^{m}{k_i}=k,\ k_i \in \mathbb{Z}\setminus\{0\},\  k_i \geq k_{i+1}\ \forall i \}\\
\\
\Lambda^{\prime}_{(k)}:=\{{t^{\prime}_0}^{k_0}{t^{\prime}_1}^{k_1}\ldots {t^{\prime}_{m}}^{k_m} | \sum_{i=0}^{m}{k_i}=k,\ k_i \in \mathbb{Z}\setminus\{0\},\  k_i \geq k_{i+1}\ \forall i \}
\end{array}
\end{equation}

\noindent In \cite{DL2} it was shown that the sets $\Lambda_{(k)}$ and $\Lambda^{\prime}_{(k)}$ are totally ordered and well ordered for all $k$ (\cite[Propositions~1 \& 2]{DL2}). Note that the sets $\Lambda$ and $\Lambda^{\prime}$ admit a natural grading: $\Lambda = \underset{k}{\oplus} \Lambda_{(k)}$ and $\Lambda^{\prime} = \underset{k}{\oplus} \Lambda^{\prime}_{(k)}$.

\smallbreak

In the rest of the section we will be using the ordering in the transitions from $\Sigma^{\prime}$ to $\Sigma$ and from $\Sigma$ to $\Lambda$.

\subsection{From $\Lambda^{\prime}$ to $\Sigma$}

In this subsection we convert monomials in the $t^{\prime}_i$'s to expressions containing the $t_i$'s. Full details and related technical lemmas are provided in \cite{DL2}. In order to simplify the expressions in this step we first introduce the following notation.

\begin{nt}\label{nt} \rm
We set $\tau_{i,i+m}^{k_{i,i+m}}:=t_i^{k_i}\ldots t^{k_{i+m}}_{i+m}$ and ${\tau^{\prime}}_{i,i+m}^{k_{i,i+m}}:={t^{\prime}}_i^{k_i}\ldots {t^{\prime}}^{k_{i+m}}_{i+m}$, for $m\in \mathbb{N}$, $k_j\neq 0$ for all $j$.
\end{nt}

 We now introduce the notion of {\it homologous words}, which is crucial for relating the sets $\Lambda^{\prime}$ and $\Lambda$ via a triangular matrix.

\begin{defn}{\cite[Definition~4]{DL2}} \rm
		We say that two words $w^{\prime}\in \Lambda^{\prime}$ and $w\in \Lambda$ are {\it homologous}, denoted $w^{\prime}\sim w$, if $w$ is
		obtained from $w^{\prime}$ by changing $t^{\prime}_i$ into $t_i$ for all $i$.
	\end{defn}
		
	\begin{ex}\rm
	The words $t^2{t_1^{\prime}}^{-1}{t_2^{\prime}}^3$ and $t^2{t_1}^{-1}{t_2}^3$ are homologous. Note also that $\Lambda^{\prime} \ni t \ \sim\ t\in \Lambda$.
	\end{ex}

In \cite[Lemma~11]{DL2} it is shown that the following relations hold in ${\rm H}_{1,n}(q)$ for $k \in \mathbb{Z}\backslash \{0 \}$:


\[
\begin{array}{lll}
{t_m^{\prime}}^{k} & = & q^{-m k}t_{m}^{k} \ + \ \sum_{i}{f_i(q) t_m^{k} w_i} \ +\ \sum_{i}{g_i(q) t^{\lambda_0}t_1^{\lambda_1}\ldots t_m^{\lambda_{m}}u_i}
\end{array}
\]

\noindent where $w_i, u_i \in {\rm H}_{m+1}(q),\ \forall i$, $\sum_{i=0}^{m}{\lambda_i}=k$ and $\lambda_i \geq 0,\ \forall i$, if $k>0$ and $\lambda_i \leq 0,\ \forall i$, if $k<0$.

\smallbreak

Using now these relations, we have that every element in $\Lambda^{\prime}$ can be expressed as a linear sum of each homologous word, the homologous word with `braiding tails' and elements in $\Sigma$ of lower order, as illustrated abstractly in Figure~\ref{fthm7}. More precisely:

\begin{thm}{\cite[Theorem~7]{DL2}}\label{convert}
		The following relations hold in ${\rm H}_{1,n}(q)$ for $k_r \in \mathbb{Z}, r= 0, \ldots, m$:
		$$
		{\tau^{\prime}}_{0,m}^{k_{0,m}}  =  q^{- \sum_{n=1}^{m}{nk_n}}\cdot {\tau}_{0,m}^{k_{0,m}}\ + \ \sum_{i}{f_i(q)\cdot {\tau}_{0,m}^{k_{0,m}}\cdot w_i} \ + \ \sum_{j}{g_j(q)\cdot \tau_j \cdot u_j},\\
		$$
		
		\noindent where $w_i, u_j \in {\rm H}_{m+1}(q)$, for all $i, j$, $\tau_j$ a monomial of the $t_i$'s such that $\tau_j < {\tau}_{0,m}^{k_{0,m}}$ for all $j$ and $f_i, g_j \in \mathbb{C}$, for all $i, j$.
	\end{thm}
	
\begin{figure}[!ht]
\begin{center}
\includegraphics[width=4in]{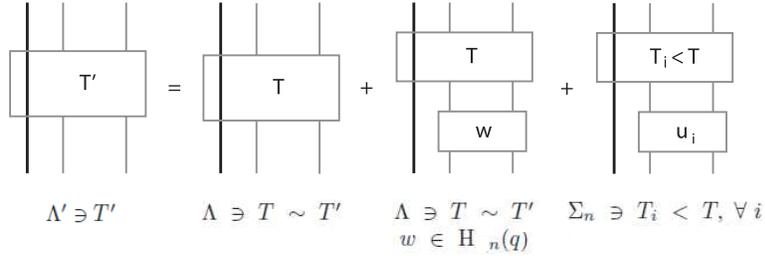}
\end{center}
\caption{Illustrating Theorem~\ref{convert}.}
\label{fthm7}
\end{figure}

	\begin{ex}\label{eg1}\rm
We shall now give examples of monomials in $\Lambda^{\prime}$ converted into sums of elements in $\Sigma$ using technical lemmas from \cite{DL2} (\cite[Lemmas~3, 4, 5, 6, 9 \& 11]{DL2}). The examples provide at the same time the motivation for the subsections that follow.
 
\bigbreak

\begin{itemize}
\item[$i.$] Consider the monomial $t{t_{1}^{\prime}}{t_2^{\prime}}^{-2} \in \Lambda^{\prime}$. We have that:

\[
\begin{array}{lll}
t{t_{1}^{\prime}}{t_2^{\prime}}^{-2} & = & q^3 \cdot tt_{1}t_2^{-2}+ q^4(q^{-1}-1)\cdot tt_{1}t_2^{-2}\cdot g_1^{-1}\ + \\
&&\\
& + & 1\cdot \left[ (q-1)^2g_1^{-1}g_2^{-1}-(q-1)^3g_1^{-2}g_2^{-1}-q^{-1}(q-1)^3g_2g_1^{-1}g_2^{-1}+q^{-1}(q-1)^3g_2^{-1}\right]\ +\\
&&\\
& + & tt_1^{-1}\cdot \left[ (q-1)(q^2-q+1)\cdot g_2^{-1}- (q-1)^2\cdot g_1g_2g_1^{-1}g_2^{-1} \right]+\\
&&\\
& + & tt_2^{-1}\cdot \left[q^2(q-1)\cdot g_2^{-1}+q(q-1)^3\cdot g_2^{-1}-q(q-1)^2\cdot g_2g_1^{-1}g_2^{-1} \right]+\\
&&\\
& + & t_1t_2^{-1}\cdot \left[ q(q-1)\cdot g_2g_1^{-1}g_2^{-1}-q(q-1)^2\cdot g_1^{-1}g_2^{-1} \right]+\\
&&\\
& + & t^{-1}t_1 \cdot \left[-(q-1)\cdot g_2g_1^{-1}g_2^{-1} - q^{-1}(q-1)^2\cdot g_1^{-1}g_2^{-1} \right]
\end{array}
\]

\noindent So, we obtain the homologous word $w=tt_{1}t_2^{-2}$, the word $w$ again followed by the braiding element $g_1^{-1}$ and terms in $\Sigma$ of order less than $w$: since either their index is less that $ind(w)$ (the terms $1$ , $tt_1^{-1}$ and $t^{-1}t_1$), or they contain gaps in the indices (the terms $tt_2^{-1}$ and $t_1t_2^{-1}$).

\smallbreak

\item[$ii.$] A simpler example is the following: 
$$t{t_1^{\prime}}^{-2}\ = \ t{t_1}^{-2}\ -\ (q^{-1}-1)\underline{{t_1}^{-1}}\cdot g_1^{-1}\ +\ (q^{-1}-1){t}^{-1}\cdot \underline{g_1}$$

\noindent We obtained the homologous word $t{t_1}^{-2}$, the element ${t}^{-1}\cdot {g_1}$ in $\Sigma$ comprising the monomial ${t}^{-1}$ with gaps in the indices, followed by `braiding tail' $g_1^{-1}$ and also the lower order term $t^{-1}g_1$ in the ${\rm H}_{n}(q)$-module $\Lambda$.
\end{itemize}

\bigbreak

In Theorem~\ref{convert} note that in the right hand side there are terms which do not belong to the set $\Lambda$. The point now is that these terms are elements in the set of bases $\Sigma$ on the Hecke algebras ${\rm H}_{1, n}(q)$, but, when we are working in $\mathcal{S}({\rm ST})$, which is the knots and links level, such elements must be considered up to conjugation by any generator of the alegbra and up to stabilization moves (recall Theorem~\ref{markov}). Topologically, conjugation corresponds to closing a mixed braid.

\end{ex}

\subsection{ From $\Sigma$ to $\Lambda$: Managing the gaps}\label{stol}

In this subsection we show how to deal with monomials in $\Sigma$ where the looping elements do not have consecutive indices. We call {\it gaps} in monomials of the $t_i$'s any gaps occurring in the indices. After managing, that is eliminating, the gaps we pass to the augmented ${\rm H}_n(q)$-module $\Lambda^{aug}$, which consists of monomials in the $t_i$'s with consecutive indices but not necessarily ordered exponents. 

\begin{defn}{\cite[Definition~3]{DLP}}\label{expsetl}
\rm
We define the sets:
$$\Lambda^{aug}_{n}\ :=\{t_0^{k_0}t_1^{k_1}\ldots t_{n}^{k_n},\ k_i \in \mathbb{Z}^*\},\quad \Lambda^{aug} := \bigcup_{n\in \mathbb{N}} \Lambda^{aug}_{n},$$ 
\noindent and the \textit{subset of level $k$} of $\Lambda^{aug}$ $\Lambda^{aug}_{(k)}$:
$$\Lambda^{aug}_{(k)}:=\{t_0^{k_0}t_1^{k_1}\ldots t_{m}^{k_m} | \sum_{i=0}^{m}{k_i}=k,\ k_i \in \mathbb{Z}^*\}.$$
\end{defn}

Note that in \cite{DLP} the set $\Lambda^{aug}$ is denoted by $L$. Obviously the set $\Lambda$ (Eq.~10) is a subset of $\Lambda^{aug}$.

\begin{nt}\rm
Whenever we talk about a module with coefficients in $\bigcup_{n\in \mathbb{N}}{\rm H}_n(q)$ we shall be denoting it by ${\rm H}_n(q)$-module.
\end{nt}

In what follows for the expressions that we obtain after appropriate conjugations we shall use the notation $\widehat{=}$.

\begin{thm}{\cite[Theorem~8]{DL2}}\label{gp}
Let $\tau$ be a monomial in the $t_i$'s with gaps in the indices. Then we have that:

$$\tau\ \widehat{=}\ \underset{i}{\sum} \tau_i\cdot \beta_i,$$ 

\noindent where $\tau_i \in \Lambda^{aug}$ and $\beta_i \in \bigcup_{n\in \mathbb{N}}{\rm H}_n(q)$ for some $n\in \mathbb{N}$ and for all $i$.
\end{thm}

Theorem~\ref{gp} is best demonstrated in the following example on a word with two gaps. Note that we underline expressions which are crucial for the next step.

\begin{ex}\rm For the 2-gap word $t^{k_0}t_1^{k_1}t_3t_{5}^2t_6^{-1}\in \Sigma$ we have:
\[
\begin{array}{llllll}
t^{k_0}\underline{t_1^{k_1}t_3}t_{5}^2t_6^{-1} & = & t^{k_0}t_1^{k_1}\underline{g_3t_2g_3}t_{5}^2t_6^{-1} & = & g_3t^{k_0}t_1^{k_1}t_2t_{5}^2t_6^{-1}g_3 & \widehat{=}\\
&&&&&\\
& \widehat{=} & t^{k_0}t_1^{k_1}\underline{t_2t_{5}^2}t_6^{-1}g_3^2 & = & t^{k_0}t_1^{k_1}t_2\underline{t_{5}}t_{5}t_6^{-1}g_3^2 & =\\
&&&&&\\
& = & t^{k_0}t_1^{k_1}t_2\underline{g_5g_4}t_3g_4g_5t_{5}t_6^{-1}g_3^2 & = & \underline{g_5g_4}t^{k_0}t_1^{k_1}t_2t_3g_4g_5t_{5}t_6^{-1}g_3^2 & 
\widehat{=}\\
&&&&&\\
&&& \widehat{=} & t^{k_0}t_1^{k_1}t_2t_3\underline{g_4g_5t_{5}}t_6^{-1}g_3^2g_5g_4 &  =\\
&&&&&\\
& = & t^{k_0}t_1^{k_1}t_2t_3\ \big[\ q^2t_3g_4g_5 & + & q(q-1)t_4g_5 + (q-1)t_5g_4\ \big] \ t_6^{-1}g_3^2g_5g_4 & =\\
&&&&&\\
& = & q^2t^{k_0}t_1^{k_1}t_2t_3^2\underline{g_4g_5t_6^{-1}}g_3^2g_5g_4 & + & q(q-1)t^{k_0}t_1^{k_1}t_2t_3t_4\underline{g_5t_6^{-1}}g_3^2g_5g_4& + \\
&&&&&\\
&  & & + & (q-1)t^{k_0}t_1^{k_1}t_2t_3t_5\underline{g_4t_6^{-1}}g_3^2g_5g_4 & = \\
&&&&&\\
 & = & q^2t^{k_0}t_1^{k_1}t_2t_3^2\underline{t_6^{-1}}g_4g_5g_3^2g_5g_4 & + & (q-1)t^{k_0}t_1^{k_1}t_2t_3\underline{t_5}t_6^{-1}g_4g_3^2g_5g_4& +\\
&&&&&\\
 & &  & + & q(q-1)t^{k_0}t_1^{k_1}t_2t_3t_4\underline{t_6^{-1}}g_5g_3^2g_5g_4 & \widehat{=}\\
\end{array}
\]

\[
\begin{array}{lllll}
\widehat{=} & q^2t^{k_0}t_1^{k_1}t_2t_3^2\underline{g_6^{-1}g_5^{-1}}t_4^{-1}g_5^{-1}g_6^{-1}g_4g_5g_3^2g_5g_4 & + & q(q-1)t^{k_0}t_1^{k_1}t_2t_3t_4\underline{g_6^{-1}}t_5^{-1}g_6^{-1} g_5g_3^2g_5g_4 & +\\
&&&&\\
&&+ & (q-1)t^{k_0}t_1^{k_1}t_2t_3\underline{g_5}{t_4}\underline{g_5}t_6^{-1}\cdot (g_4g_3^2g_5g_4)& = \\
&&&&\\
= & q^2g_6^{-1}g_5^{-1}t^{k_0}t_1^{k_1}t_2t_3^2t_4^{-1}g_5^{-1}g_6^{-1}g_4g_5g_3^2g_5g_4 & + & q(q-1)g_6^{-1}t^{k_0}t_1^{k_1}t_2t_3t_4t_5^{-1}g_6^{-1} g_5g_3^2g_5g_4 & +\\
&&&&\\
&&+ & (q-1){g_5}t^{k_0}t_1^{k_1}t_2t_3{t_4}t_6^{-1}{g_5}g_4g_3^2g_5g_4 & \widehat{=} \\
&&&&\\
\widehat{=} & q^2t^{k_0}t_1^{k_1}t_2t_3^2t_4^{-1}g_5^{-1}g_6^{-1}g_4g_5g_3^2g_5g_4g_6^{-1}g_5^{-1} & + & q(q-1)t^{k_0}t_1^{k_1}t_2t_3t_4t_5^{-1}g_6^{-1} g_5g_3^2g_5g_4g_6^{-1} & + \\
&&&&\\
&& + & (q-1)t^{k_0}t_1^{k_1}t_2t_3{t_4}\underline{t_6^{-1}}{g_5}\cdot (g_4g_3^2g_5g_4{g_5}) & = \\
&&&&\\
= & q^2t^{k_0}t_1^{k_1}t_2t_3^2t_4^{-1}g_5^{-1}g_6^{-1}g_4g_5g_3^2g_5g_4g_6^{-1}g_5^{-1} & +& q(q-1)t^{k_0}t_1^{k_1}t_2t_3t_4t_5^{-1}g_6^{-1} g_5g_3^2g_5g_4g_6^{-1} & +\\
&&&&\\
&& + & (q-1)t^{k_0}t_1^{k_1}t_2t_3{t_4}\underline{g_6^{-1}}t_5^{-1}g_6^{-1}{g_5}g_4g_3^2g_5g_4{g_5}&  \widehat{=}\\
&&&&\\
\widehat{=} & q^2t^{k_0}t_1^{k_1}t_2t_3^2t_4^{-1}g_5^{-1}g_6^{-1}g_4g_5g_3^2g_5g_4g_6^{-1}g_5^{-1} & + & q(q-1)t^{k_0}t_1^{k_1}t_2t_3t_4t_5^{-1}g_6^{-1} g_5g_3^2g_5g_4g_6^{-1} &+\\
&&&&\\
&& + & (q-1)t^{k_0}t_1^{k_1}t_2t_3{t_4}t_5^{-1}g_6^{-1}{g_5}g_4g_3^2g_5g_4{g_5}{g_6^{-1}}
\end{array}
\]

\end{ex}

\subsection{From the ${\rm H}_n(q)$-module $\Sigma$ to the ${\rm H}_n(q)$-module $\Lambda$: Ordering the exponents}

By Theorem~\ref{gp} we have to deal with elements in $\Lambda^{aug}$, where the looping generators have consecutive indices but their exponents are not in decreasing order, followed by `braiding tails'. We show that these elements are conjugate equivalent to sums of elements in the $\textrm{H}_n(q)$-module $\Lambda$, namely, elements in $\Lambda$ followed by `braiding tails'.

\begin{thm}{\cite[Theorem~9]{DL2}}\label{exp}
For an element in the ${\rm H}_n(q)$-module $\Lambda^{aug}$ we have that:

$$\tau_{0,m}^{k_{0,m}}\cdot w\ \widehat{=}\ \sum_{j}{\tau_{0,j}^{\lambda_{0,j}}\cdot w_j},$$

\noindent where $\tau_{0,j}^{\lambda_{0,j}} \in \Lambda$ and $w, w_j \in \bigcup_{n\in \mathbb{N}}{\rm H}_n(q)$ for all $j$.
\end{thm}

\begin{ex}\rm
Consider the element $tt_1^2t_2^3 \in \Lambda^{aug}$ and apply Theorem~\ref{exp} on the first `bad' exponent occurring in the word, starting from right to left. In that way we obtain a word with one less `bad' exponent, so applying Theorem~\ref{exp} again we obtain elements in the set $\Lambda$ followed by braiding tails. More precisely:
$$
tt_1^2t_2^3 \ \widehat{=} \ a\cdot t^3t_1^2t_2\cdot u_1\ +\ b\cdot t^2t_1^2t_2^2\cdot u_2\ +\ c\cdot t^4t_1t_2\cdot u_3
$$
\noindent where $u_1, u_2, u_3 \in {\rm H}_{3}(q)$, for all $i$ and $a, b, c \in \mathbb{C}[q^{\pm 1}]$. 

\end{ex}

\subsection{From the ${\rm H}_n(q)$-module $\Lambda$ to $\Lambda$: Eliminating the tails}

So far we have seen how to convert elements in the basis $\Lambda^{\prime}$ to sums of elements in $\Sigma$ and then, using conjugation, how these elements are expressed as sums of elements in the $\textrm{H}_n(q)$-module $\Lambda$. We now present results on how using conjugation and stabilization moves all these elements in the $\textrm{H}_n(q)$-module $\Lambda$ are expressed as sums of elements in the set $\Lambda$ with scalars in the field $\mathbb{C}$. We will use the symbol $\simeq$ when a stabilization move is performed and $\widehat{\simeq}$ when both stabilization moves and conjugation are used. More precisely, in \cite{DL2} we prove the following:

\begin{thm}{\cite[Theorem~10]{DL2}} \label{tails}
For a word in the ${\rm H}_n(q)$-module, $\Lambda$ we have:

$$\tau_{0,m}^{k_{0,m}}\cdot w_n\ \widehat{\simeq}\  \sum_{j}{f_j(q,z)\cdot \tau_{0,u_j}^{v_{0,u_j}}},$$

\noindent such that $\sum{v_{0,u_j}}=\sum{k_{0,m}}$, $\tau_{0,u_j}^{v_{0,u_j}} \in \Lambda^{aug}$ and $\tau_{0,u_j}^{v_{0,u_j}} < \tau_{0,m}^{k_{0,m}}$ for all $j$.
\end{thm}

\begin{ex}\rm In this example we demonstrate how to eliminate the `braiding tail' in a word.
\[
\begin{array}{lcl}
t^{3}\underline{t_1^2}t_{2}^{-1}g_1^{-1} & = & t^{3}t_1t_{2}^{-1}\underline{t_1g_1^{-1}}\ = \ t^{3}t_1t_{2}^{-1} g_1 \underline{t} \ \widehat{=} \ t^4\underline{t_1}t_{2}^{-1} g_1\ =\ t^4t_{2}^{-1}\underline{t_1 g_1}\ =\\
&&\\
& = &  (q-1)  t^4\underline{t_1}t_{2}^{-1}\ +\ q t^4t_{2}^{-1} g_1 \underline{t}\ \widehat{=}\ (q-1) t^5\underline{t_2^{-1}}g_1^{2}\ +\ qt^5 \underline{t_{2}^{-1}} g_1\ =\\
&&\\
& = & (q-1) t^5t_1^{-1}g_2^{-1}g_1^{2}g_2^{-1}\ +\ q t^5t_1^{-1}g_2^{-1}g_1g_2^{-1}.\\
\end{array}
\]

\noindent We have that:
\[
\begin{array}{lcl}
g_2^{-1}g_1^2g_2^{-1} & = & q^{-2}(q-1)g_1g_2g_1\ -\ (q^{-1}-1)^2g_2g_1\ -\ (q^{-1}-1)^2 g_1g_2\ +\\
&&\\
& + & (q-1)(q^{-1}-1)^2g_1\ +\ q(q^{-1}-1)g_2^{-1}\ +\ 1,\\
&&\\
g_2^{-1}g_1g_2^{-1} & = & q^{-2}g_1g_2g_1\ +\ q^{-1}(q^{-1}-1)g_2g_1\ +\  q^{-1}(q^{-1}-1) g_1g_2\ +\\
&&\\
& + & (q^{-1}-1)^2g_1,\\
\end{array}
\]

\noindent and so
\[
\begin{array}{rcl}
(q-1)\cdot t^5t_1^{-1}g_2^{-1}g_1^{2}g_2^{-1} & \widehat{\simeq} & \left( (q-1) + q^{-1}(q-1)^3 \right)\cdot t^5t_1^{-1} - q^{-3}(q^{-1}-1)^3z^2\cdot t^4 +\\
&&\\
& + & 3q^{-3}(q-1)^4z\cdot t^4 - q^{-1}(q-1)^2z\cdot t^4 - q^{-3}(q-1)^5\cdot t^4,\\
&&\\
q\cdot t^5t_1^{-1}g_2^{-1}g_1g_2^{-1} & \widehat{\simeq} & z \cdot t^5t_1^{-1} + q^{-1}(q^{-1}-1)z^2\cdot t^4 + 2(q^{-1}-1)^2z\cdot t^4 +\\
&&\\
& + &  q(q^{-1}-1)^3\cdot t^4.\\
\end{array}
\]

\end{ex}

\begin{remark}\label{rm4}\rm
This is a long procedure, since eliminating the tails will give rise to gaps in the indices again. Recall however that when managing gaps in the indices and ordering the exponents of the $t_i$'s we also obtain `braiding tails'. This is a long procedure which, as shown in \cite{DL2} using complex and technical inductions, this procedure eventually terminates and only elements in $\Lambda$ remain. This means that $\Lambda$ is a generating set for $\mathcal{S}({\rm ST})$. This procedure is abstractly demonstrated in Figure~\ref{alg}. In the figure, we start with a word $\tau^{\prime}$ in the old basis of $\mathcal{S}({\rm ST})$, $\Lambda^{\prime}$, and applying Theorems~8--10 we end up with a linear combination of the homologous word $\tau\in \Lambda$ and of elements $\lambda_i\in \Lambda$ smaller than $\tau$.
\end{remark}

\begin{figure}
\begin{center}
\includegraphics[width=5.2in]{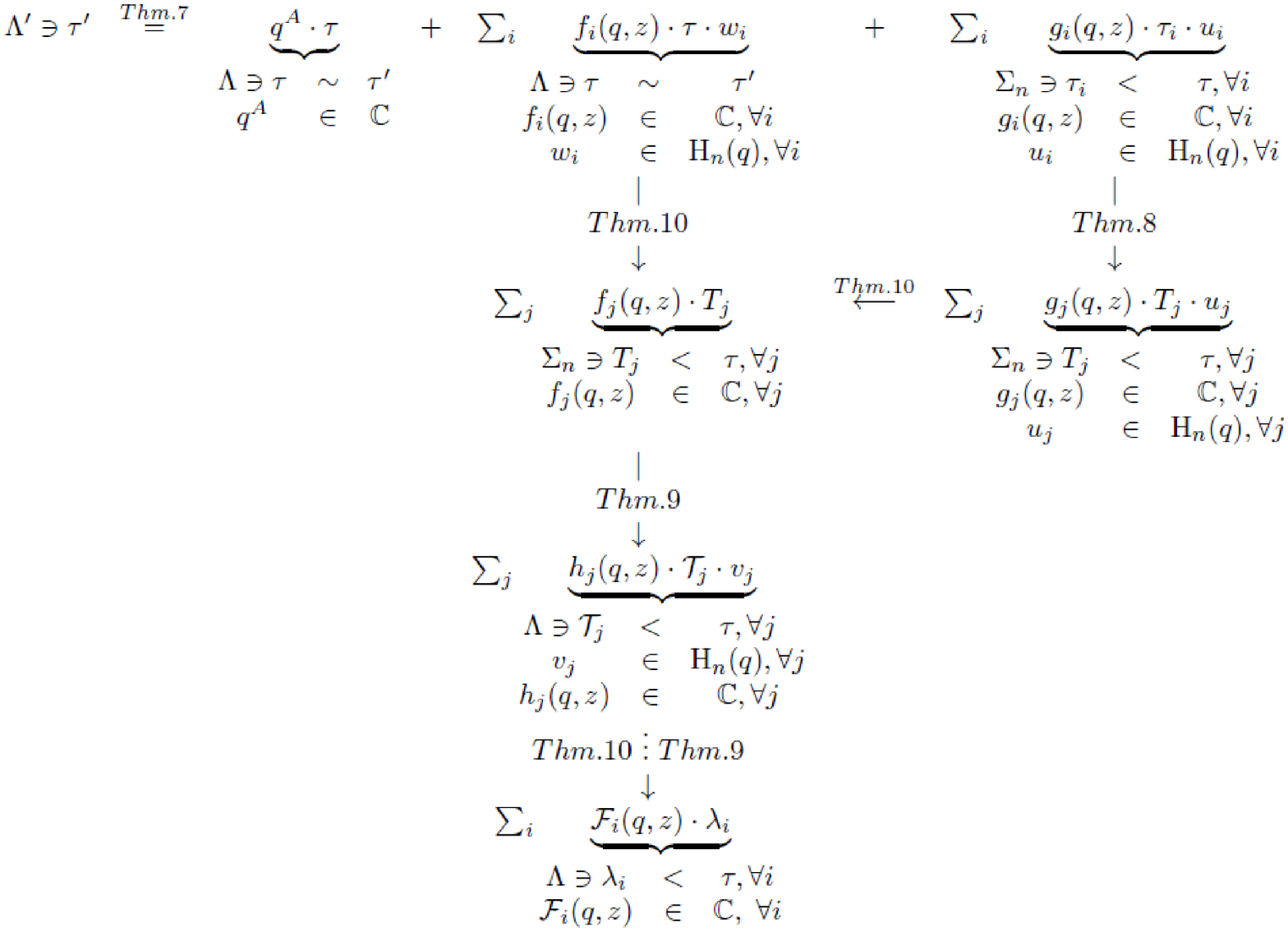}
\end{center}
\caption{From $\Lambda^{\prime}$ to $\Lambda$.}
\label{alg}
\end{figure}

\subsection{The infinite matrix}

With the orderings given in Definition~\ref{order}, in \cite{DL2} we showed that the infinite matrix converting elements of the basis $\Lambda^{\prime}$ of $\mathcal{S}({\rm ST})$ to elements of the set $\Lambda$ is a block diagonal matrix, where each block corresponds to a subset of $\Lambda^{\prime}$ of level $k$ and it is an infinite lower triangular matrix with invertible elements in the diagonal. This constitutes our strategy for proving Theorem~\ref{newbasis}. More precisely, fixing the level $k$ of a subset of $\Lambda^{\prime}$, the proof of Theorem~\ref{newbasis} is based on the following:

\smallbreak

\begin{itemize}
\item[(1)] A monomial $w^{\prime} \in \Lambda_{(k)}^{\prime} \subseteq \Lambda^{\prime}$ can be expressed as linear combinations of elements in $\Lambda_{(k)} \subseteq \Lambda$, $v_i$, followed by monomials in $\textrm{H}_n(q)$, with scalars in $\mathbb{C}$ such that there exists $j: v_j=w\sim w^{\prime}$.
\smallbreak
\item[(2)] Applying conjugation and stabilization moves on all $v_i$'s results in elements $u_i$ in $\Lambda_{(k)}$, such that $u_i < v_i$ for all $i$.
\smallbreak
\item[(3)] The coefficient of $w$ is an invertible element in $\mathbb{C}$.
\smallbreak
\item[(4)] $\Lambda_{(k)} \ni w < u \in \Lambda_{(k+1)}$.
\smallbreak
\item[(5)] Using this infinite diagonal matrix, in \cite[Theorem~11]{DL2} we showed that the set $\Lambda$ is linearly independent. Hence, using the above and Remark~\ref{rm4}, $\Lambda$ forms a basis for the HOMFLYPT skein module of ST.
\end{itemize}

\section{Topological steps toward $\mathcal{S}\left(L(p,1)\right)$}

We now return to our initial aim, that is, the computation of the HOMFLYPT skein module of a lens space $L(p,1)$. As explained in the Introduction, in order to compute $\mathcal{S}(L(p,1))$ we need to normalize the invariant $X$ (recall Theorem~5) by forcing it to satisfy all possible braid band moves ({\it bbm}), recall Eq.~1. At this point the reader should recall the discussion in the Introduction culminating to Equation~2. In order to simplify this system of equations, in \cite{DLP} we first show that performing a {\it bbm} on a mixed braid in $B_{1, n}$ reduces to performing {\it bbm}'s on elements in the canonical basis, $\Sigma_n^{\prime}$, of the algebra ${\rm H}_{1,n}(q)$ and, in fact, on their first moving strand. We then reduce the equations obtained from elements in $\Sigma^{\prime}$ to equations obtained from elements in $\Sigma$. In order now to reduce further the computation to elements in the basis $\Lambda$ of $\mathcal{S}({\rm ST})$, we first recall that elements in $\Sigma$ consist in two parts: a monomial in $t_i$'s with possible gaps in the indices and unordered exponents, followed by a `braiding tail' in the basis of ${\rm H}_n(q)$. So, we first manage the gaps in the indices of the looping generators of elements in $\Sigma$, obtaining elements in the augmented ${\rm H}_{n}(q)$-module $\Lambda^{aug}$ (followed by `braiding tails'). Note that the performance of {\it bbm}'s is now considered to take place on any moving strand. We then show that the equations obtained from elements in the ${\rm H}_{n}(q)$-module $\Lambda^{aug}$ by performing {\it bbm}'s on any strand are equivalent to equations obtained from elements in the ${\rm H}_{n}(q)$-module $\Lambda$ by performing {\it bbm}'s on any strand (ordering the exponenets in the $t_i$'s). We finally eliminate the `braiding tails' from elements in the ${\rm H}_{n}(q)$-module $\Lambda$ and reduce the computations to the set $\Lambda$, where the {\it bbm}'s are performed on any moving strand (see \cite{DLP}). Thus, in order to compute $\mathcal{S}(L(p,1))$, it suffices to solve the infinite system of equations obtained by performing {\it bbm}'s on any moving strand of elements in the set $\Lambda$.

In this section we present the above steps in more details. The procedure is similar to the one described in \cite{DL2} (see \S~2 in this paper), but now we do it simultaneously before and after the performance of a braid band move.

\subsection{Reducing computations from $B_{1, n}$ to $\Sigma_n^{\prime}$}\label{lp1}

We now show that it suffices to perform {\it bbm}'s on elements in the linear basis of ${\rm H}_{1,n}(q)$, $\Sigma^{\prime}$. As already mentioned, this is the first step in order to restrict the performance of {\it bbm}'s only on elements in the basis $\Lambda$. We first recall that by the Artin combing we can write words in $B_{1, n}$ in the form $\tau^{\prime}\cdot w$, where $\tau^{\prime}$ is a monomial in the $t_i^{\prime}$'s and $w\in B_n$ (Remark~1). We then note the following:

\begin{lemma}\label{ArtComb}
Braid band moves are interchangeable with the Artin combing.
\end{lemma}

\begin{proof}
Let $d\in B_{1, n}$. Then, the proof is clear from the following diagram:
$$
\begin{CD}
B_{1, n} \ni & d @>(\pm)(p,1) bbm>> t^p d^{\prime}\cdot \sigma_1^{\pm1}\\
&@VVArtin\ combingV @VVArtin\ combingV\\
& \tau_1^{\prime}\cdot w @>(\pm)(p,1) bbm>> t^p \tau_2^{\prime}\cdot \sigma_1^{\pm1}
\end{CD}
$$

\end{proof}

\begin{lemma}{\cite[Lemma~2]{DLP}}\label{lbbm'skein}
Braid band moves and the quadratic relation (skein relation) are interchangeable. 
\end{lemma}

\begin{proof}
By Lemma~\ref{ArtComb}, a word in $B_{1, n}$ can be assumed in the form $\tau_1^{\prime}\cdot w$, where $\tau_1^{\prime}$ is a monomial in $t_i^{\prime}$'s and $w\in B_{n}$. Seen as a monomial in ${\rm H}_n(q)$ and applying the quadratic relation, the element $w$ can be written as a sum: $w=\sum_{i=1}^{n}{f_{i}(q)w_{i}}$, where the $w_i$'s are words in ${\rm H}_n(q)$ in canonical form and the $f_{i}(q)$ are expressions in $\mathbb{C}$ for all $i$. We perform a braid band move on the element $\tau_1^{\prime}\cdot w\in B_{1, n}$ and we obtain:

\[
\tau_1^{\prime}\cdot w \ \xrightarrow[bbm]{(\pm)(p,1)} \ t^p\tau_2^{\prime}\cdot w_+\sigma_1^{\pm1},
\]

\noindent where $w_+\in B_{n+1}$ is the same word as $w$ but with all indices shifted by one. Hence, on the algebra level we have $w_i\in {\rm H}_{n}(q)$ and $w_+=\sum_{i=1}^{n}{f_{i}(q)w_{i_+}}\in {\rm H}_{n+1}(q)$. So:

\[
\tau_1^{\prime}\cdot w \ = \ \tau_1^{\prime}\cdot \sum_{i=1}^{n}{f_{i}(q)w_{i}}\ \xrightarrow[bbm]{(\pm)(p,1)} \ t^p\tau_2^{\prime}\cdot \sum_{i=1}^{n}{f_{i}(q)w_{i_+}} g_1^{\pm1}.
\]

\noindent On the other hand, for each mixed braid $\tau_1^{\prime}\cdot w_i$ we also have:

\[
\tau_1^{\prime}\cdot w_i \ \xrightarrow[bbm]{(\pm)(p,1)}\ t^p\tau_2^{\prime}\cdot w_{i_+} \sigma_1^{\pm1}\ \forall\ i,
\]

\noindent and thus, on the algebra level we finally obtain:
\[
\tau_2^{\prime}\cdot \sum_{i=1}^{n}{f_{i}(q)w_{i}} g_1^{\pm1}\ =\ t^p\tau_2^{\prime}\cdot w_{+} g_1^{\pm1}.
\]

\noindent That is, the following diagram commutes:
$$
\begin{CD}
\tau_1^{\prime} \cdot w @>(\pm)(p,1) bbm>> t^p\tau_2^{\prime}\cdot w_+ \sigma_1^{\pm1}\\
@VVquadraticV @VVquadraticV\\
\sum_i f_i(q)\tau_1^{\prime}\cdot w_i @>(\pm)(p,1) bbm>> \sum_i f_i(q) t^p\tau_2^{\prime}\cdot w_{i_+} g_1^{\pm1}
\end{CD}
$$

\noindent See also Figure~\ref{bbm'skein}. So the proof is concluded.

\end{proof}

\begin{figure}
\begin{center}
\includegraphics[width=4.5in]{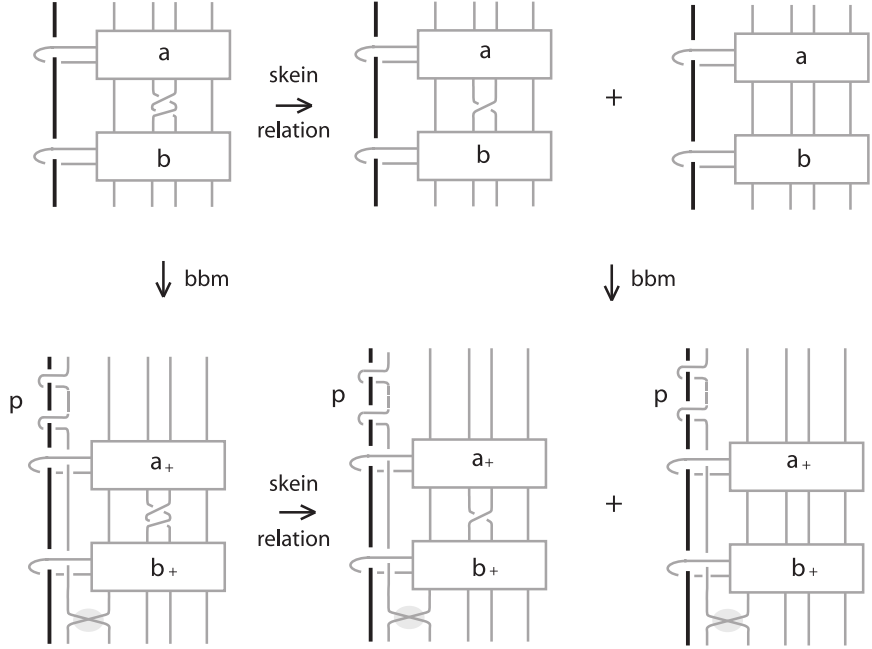}
\end{center}
\caption{Proof of Lemma~\ref{lbbm'skein}.}
\label{bbm'skein}
\end{figure}

Furthermore we have:

\begin{lemma}\label{sprbbm}
The procedure of bringing a looping monomial in the $t_i^{\prime}$'s in the form of elements in the sets $\Sigma_n^{\prime}$ of Theorem~2 is consistent with the braid band moves.
\end{lemma}

\begin{proof}
Let $\tau_1^{\prime}\cdot w$ be an element in $B_{1, n}$ and $t^p\tau_2^{\prime}\cdot w_+\cdot \sigma_1^{\pm 1}$ the result of a performance of a {\it bbm} on $\tau_1^{\prime}\cdot w$. We now follow \cite[Proposition~2 \& Theorem~1]{La4} so as to order the indices of the monomials $\tau_1^{\prime}$ and $\tau_2^{\prime}$ in the $t_i^{\prime}$'s (before and after the performance of the {\it bbm}). Cabling the first moving strand coming from the performance of the {\it bbm} with the fixed strand of the mixed braid, and viewing this cable as one thickened strand, we have that the procedure we follow to order the indices of the $t_i^{\prime}$'s is identical before and after the performance of the {\it bbm} and this concludes the proof.
\end{proof}

Using Lemmas~\ref{ArtComb}, \ref{lbbm'skein} and \ref{sprbbm} we now have the following:

\begin{prop}{\cite[Proposition~1]{DLP}}\label{bbm's'}
It suffices to consider the performance of braid band moves on the first strand of only elements in the set $\Sigma^{\prime}$.
\end{prop}

\begin{proof}
By the Artin combing, any $d\in B_{1,n}$ can be written in the form $\tau^{\prime}\cdot w$, where $\tau^{\prime}$ is a monomial in $t_i^{\prime}$'s and $w \in B_n$.
By Lemma~\ref{ArtComb} we have that:
\[
\begin{array}{rcll}
X_{\widehat{\tau^{\prime}\cdot w}} & = & X_{\widehat{t^p \tau^{\prime \prime}\cdot \sigma_1^{\pm 1} \cdot w_{+}}} &\overset{(Lemma~\ref{lbbm'skein})}{\Rightarrow} \\
&&&\\
\sum_i{A_i \cdot X_{\widehat{\tau^{\prime}\cdot w_i}}} & = & \sum_i{A_i \cdot X_{\widehat{t^p \tau^{\prime \prime}\cdot \sigma_1^{\pm 1} \cdot w_{i_{+}}}}}, & \\
\end{array}
\]

\noindent where $w_i$ are words in reduced form in ${\rm H}_{n}(q),\ \forall i$ and $A_i \in \mathbb{C}$. Then, we order the indices of the monomials in the $t_i^{\prime}$'s before and after the performance of the {\it bbm}. The result follows from Lemma~\ref{sprbbm}.
\end{proof}

\subsection{From the set $\Sigma^{\prime}$ to the set $\Sigma$}

We shall now show that it suffices to perform {\it bbm}'s on elements in the linear bases $\Sigma$ of the algebras ${\rm H}_{1,n}(q)$, which includes as a proper subset the basis $\Lambda$ of $\mathcal{S}({\rm ST})$, described in \S~\ref{lamb}. Indeed, let $\tau^{\prime}\cdot w \in \Sigma^{\prime}$ as above. Then, by (8) we have:
\[
\begin{matrix}
\tau^{\prime}\cdot w & = & (t^{k_0} {t^{\prime}_1}^{k_1 } \ldots {t^{\prime}_m}^{k_m})\cdot w & {=} &\underset{\tau}{\underbrace{ t^{k_0} (t_1g_1^{-2})^{k_1} \ldots ({t_mg_m^{-1}\ldots g_2^{-1} g_1^{-2} g_2^{-1} \ldots g_m^{-1}})^{k_m}}}\cdot w\ =\\
  & = & \tau \cdot w,  &  &    \\
\end{matrix}
\]

\noindent see top row of Fig.~18.

\smallbreak

We then perform a {\it bbm} on the first moving strand of both $\tau^{\prime}\cdot w$ and
$\tau \cdot w$ (see bottom row of Fig.~18) and we cable the new parallel strand together with the surgery strand. Denote the result of cabling the new strand appearing after the performance of the {\it bbm} with the fixed strand as $cbl(ps)$. Then:
\[
\begin{matrix}
\tau^{\prime}\cdot w & \overset{bbm}{\rightarrow} & cbl(ps) \cdot \tau^{\prime}\cdot w \cdot g_1^{\pm 1}\\
\parallel & & \parallel \\
\tau\cdot w & \overset{bbm}{\rightarrow} & cbl(ps) \cdot \tau\cdot w \cdot g_1^{\pm 1}\\
\end{matrix}
\]
So: $X_{\widehat{\tau^{\prime}\cdot w}}\ =\ X_{\widehat{bbm(\tau^{\prime}\cdot w)}}\ \Leftrightarrow \ X_{\widehat{\tau\cdot w}}\ =\ X_{\widehat{bbm(\tau\cdot w)}}$.
But since $\tau\cdot w \in {\rm H}_{1,n}(q)$ for some $n\in \mathbb{N}$, we can express $\tau\cdot w$ as a sum of elements in the linear basis $\Sigma_n$ of ${\rm H}_{1,n}(q)$, that is, $\tau\cdot w = \sum_{i}{a_i T_i \cdot w_i}$, where $a_i \in \mathbb{C}$ and $T_i\cdot w_i \in \Sigma_n$, for all $i$ and $T_i$ is a monomial in the $t_j$'s with possible gaps in the indices and
unordered exponents. Then, by Theorem~5:
\[
\begin{array}{lllll}
X_{\widehat{\tau\cdot w}}\ = \ X_{\widehat{bbm(\tau\cdot w)}} & \Leftrightarrow & tr(\tau\cdot w) & = & \Delta \cdot tr\left(cbl(ps) \tau \cdot w \cdot g_1^{\pm1} \right)\\
& \Leftrightarrow & \sum_{i}{a_i\, tr(T_i\cdot w_i)} & = & \Delta \cdot \sum_{i}{a_i\, tr\left(\cdot cbl(ps)T_i\cdot w_i\cdot g_1^{\pm1} \right)}. \\
\end{array}
\]

We conclude that:

\[
\begin{matrix}
\tau^{\prime}\cdot w & \overset{bbm}{\rightarrow} & cbl(ps) \cdot \tau^{\prime}\cdot w \cdot g_1^{\pm 1} & (\ast) \\
\parallel & & \parallel &\\
\tau\cdot w & \overset{bbm}{\rightarrow} & cbl(ps) \cdot \tau\cdot w \cdot g_1^{\pm 1}&\\
\parallel & & \parallel &\\
\sum_{i}{a_i\cdot T_i\cdot w_i} & \overset{bbm}{\rightarrow} & \sum_{i}{a_i\cdot t^p{T_i}_{+}\cdot {w_i}_{+}g_1^{\pm1}}& (\ast \ast)\\
\end{matrix}
\]

The main ponts in the above procedure is that when performing a {\it bbm}, the looping generators $t_i$ (in some ${\rm H}_{1, n}$) remain formally the same (in some ${\rm H}_{1, n+1}$). This implies that after the {\it bbm} and the cabling operation, the words $\tau^{\prime}\cdot w$ and $\tau\cdot w$ remain formally the same. Furthermore, when converting the words $\tau \cdot w$ and $bbm(\tau\cdot w)$ into linear sums of elements in the bases $\Sigma_n$ and $\Sigma_{n+1}$ respectively, the coefficients $a_i\in \mathbb{C}$ in $(\ast \ast)$ remain the same.

\smallbreak

The above are summarized in the following proposition:

\begin{prop}{\cite[Proposition~2]{DLP}}\label{picprop2}
The equations

\begin{equation}\label{eq1}
X_{\widehat{T^{\prime}\cdot w}} \ =\ X_{\widehat{bbm_1(T^{\prime}\cdot w)}}
\end{equation}

\noindent result from equations of the form 

\begin{equation}\label{eq2}
X_{\widehat{T\cdot w}} \ =\ X_{\widehat{t^p T_+ \cdot w_+} \cdot g_1^{\pm 1}},
\end{equation}

\noindent where $T^{\prime}\cdot w \in \Sigma^{\prime}$ and $T\cdot w \in \Sigma$.
\end{prop}

\begin{figure}
\begin{center}
\includegraphics[width=2.2in]{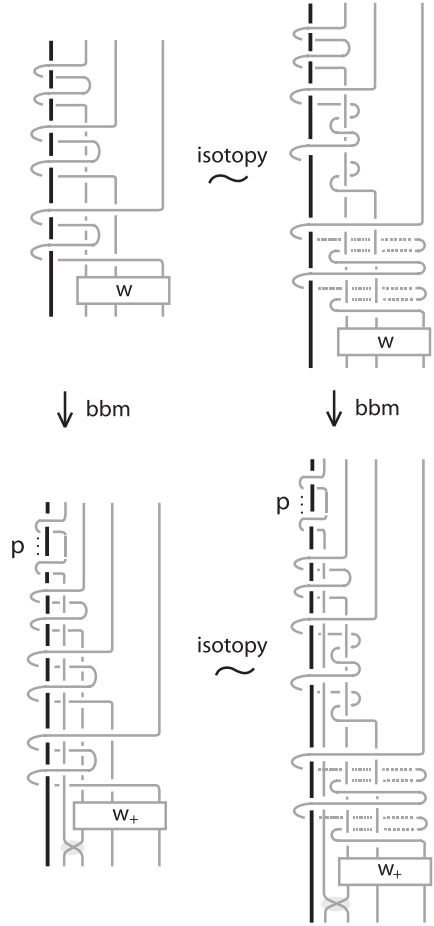}
\end{center}
\caption{Proof of Proposition~\ref{picprop2}. }
\label{prop2a}
\end{figure}

\subsection{From the set $\Sigma$ to the ${\rm H}_n(q)$-module $\Lambda^{aug}$: Managing the gaps}

As mentioned in \S~2, a word in $\Sigma$ is a monomial in the $t_i$'s followed by a `braiding tail', a monomial in the $g_i$'s. This `braiding tail' is a word in the algebra ${\rm H}_n(q)$ and the monomial in the $t_i$'s may have gaps in the indices. Using the ordering relation given in Definition~\ref{order} and conjugation these gaps are managed by showing that a monomial in the $t_i$'s can be expressed as a sum of monomials in the $t_i$'s with consecutive indices, which are of less order than the initial word and which are followed by `braiding tails' (Theorem~8). Note that the exponents of the $t_i$'s are in general not ordered, so these end monomials do not necessarily belong to the basis $\Lambda$ of $\mathcal{S}({\rm ST})$. In order to restrict the {\it bbm}'s only on elements in $\Lambda$, we need first to augment the set $\Lambda$. So, as a first step we consider the augmented set $\Lambda^{aug}$ in $\mathcal{S}({\rm ST})$ that contains monomials in the $t_i$'s with consecutive indices but arbitrary exponents. Note that, due to the presence of `braiding tails', the set $\Lambda^{aug}$ is considered as an ${\rm H}_n(q)$-module.

We now proceed with showing that Eqs~(\ref{eq2}) (Proposition~2) reduce to equations of the same type, but with elements in the set $\Lambda^{aug}$. For that, we need the following lemma about the monomial $t_1^k \in \Sigma\, \backslash\, \Lambda^{aug}$, which serves as the basis of the induction applied for proving the main result of this section, Proposition~\ref{imp}.

\begin{lemma}{\cite[Lemma~3]{DLP}}\label{gap}
The equations $X_{\widehat{t_1^{k}}}\ =\ X_{\widehat{t^pt_2^{k}\sigma_1^{\pm 1}}}$ are equivalent to the equations
\[
\begin{array}{llll}
X_{\widehat{{t}^{u_0}{t_1}^{u_1}}} & = & X_{\widehat{t^pt_1^{u_0}t_2^{u_1}\sigma_1^{\pm 1}}}, &\forall\ u_0,u_1<k\ :\ u_0 +u_1\ =\ k, \\
X_{\widehat{t^{k}}} & = & X_{\widehat{t^pt_1^{k}\sigma_1^{\pm 1}}}, & {bbm\ on\ 1st\ strand},\ t^k\in \Sigma_2,\\
X_{\widehat{t^{k}}} & = & X_{\widehat{t^pt_1^{k} \sigma_2\sigma_1^{\pm 1}}\sigma_2^{-1}} & {bbm\ on\ 2nd \ strand},\ t^k\in \Sigma_2.\\
\end{array}
\]
\end{lemma}

Indeed, in \cite{DLP} we prove that:

\[
\begin{array}{lcl}
t_1^{k} & \overset{bbm}{\underset{1^{st} str.}{\longrightarrow}} & t^{p}t_2^{k}\sigma_1^{\pm 1}\\
\widehat{\simeq} & & \widehat{\simeq}\\
{\scriptstyle (q-1)^2\sum_{j=0}^{k-2}\sum_{\phi=0}^{k-2-j}q^{j+\phi}t^{j+1+\phi}t_1^{k-1-j-\phi}} & \overset{bbm}{\underset{1^{st} str.}{\longrightarrow}} & {\scriptstyle (q-1)^2\sum_{j=0}^{k-2}\sum_{\phi=0}^{k-2-j}q^{j+\phi}t^pt_1^{j+1+\phi}t_2^{k-1-j-\phi}\sigma_1^{\pm 1}} \\
 (q-1)(k-1)q^{k-1}zt^{k} & \overset{bbm}{\underset{1^{st} str.}{\longrightarrow}} &  (q-1)(k-1)q^{k-1}zt^pt_1^{k}\sigma_1^{\pm 1}\\
 (q-1)q^{k-1}zt^{k} & \overset{bbm}{\underset{1^{st} str.}{\longrightarrow}} & (q-1)q^{k-1}zt^pt_1^{k}\sigma_1^{\pm 1} \\
 q^kt^k & \overset{bbm}{\underset{2^{nd} str.}{\longrightarrow}} & q^{k}zt^pt_1^{k}\sigma_2\sigma_1^{\pm 1} \sigma_2^{-1}.\\
\end{array}
\]

\noindent which captures the main idea of the proof of Lemma~\ref{gap}.

\smallbreak

Let now $\tau_{gaps}$ denote a word containing gaps in the (ordered) indices but not starting with a gap. When managing the gaps, the first part of the word (before the first gap) remains intact after managing the gaps and the same carries through after the performance of a {\it bbm} on the first moving strand. That is, the following diagram commutes:
\[
\begin{matrix}
\tau_{gaps}\cdot w & \underset{bbm}{\overset{1^{st} str.}{\longrightarrow}} & t^p  {\tau_{gaps}}_{+} \cdot w_+ g_1^{\pm 1} \\
\mid & & \mid \\
 man.\ gaps     & & man.\ gaps\\
\downarrow & & \downarrow \\
\sum_{i} A_i \tau_{i}\cdot w_i & \underset{bbm}{\overset{1^{st} str.}{\longrightarrow}} & \sum_{i} A_i t^p \tau_{i_+}\cdot w_{i_+}g_1^{\pm1} \\
\end{matrix}
\]
\noindent where $\tau_{gaps}\cdot w \in \Sigma$ and $\tau_i \in \Lambda^{aug}$, for all $i$.

In the case where the word $\tau_{gaps}\cdot w \in \Sigma$ starts with a gap, in \cite{DLP} it is shown that equations obtained from $\tau\cdot w$ are equivalent
to equations obtained from elements $\tau_i\cdot w_i \in \Sigma$, where $\tau_i$ are monomials in the $t_i$'s {\it not} starting with a gap, {\it but with the {\it bbm} performed on any strand} (see Fig.~\ref{exgaps}).

\smallbreak

The above are summarized in the following proposition:

\begin{prop}{\cite[Proposition~3]{DLP}}\label{imp}
In order to obtain an equivalent infinite system to the one obtained from elements in $\Sigma$ by performing braid band moves on the first moving strand, it suffices to consider monomials in $\Lambda^{aug}$ followed by braiding tails in ${\rm H}_n(q)$ and perform braid band moves on any moving strand.
\end{prop}

\begin{figure}
\begin{center}
\includegraphics[width=2.7in]{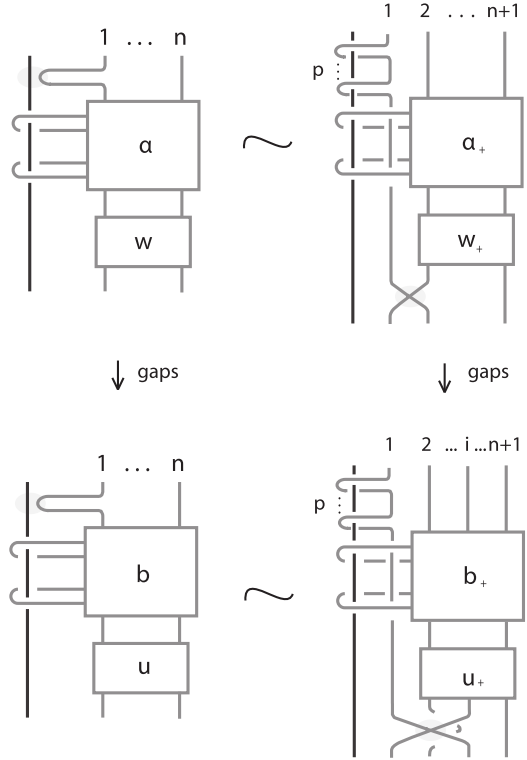}
\end{center}
\caption{{\it Bbm}'s before and after managing the gaps.}
\label{exgaps}
\end{figure}

\subsection{From the ${\rm H}_n(q)$-module $\Lambda^{aug}$ to the ${\rm H}_n(q)$-module $\Lambda$: Ordering the exponents}

The monomials in the $t_i$'s that we obtain after managing the gaps are not elements in the set $\Lambda$, since the exponents of the loop generators are not necessarily ordered. We now order the exponents of the $t_i$'s and we show that equations obtained from elements in the ${\rm H}_n(q)$-module $\Lambda^{aug}$ reduce to equations obtained from elements in the ${\rm H}_n(q)$-module $\Lambda$.

\smallbreak

The procedure we follow is similar to the one described in \cite{DL2}, but, as we mentioned earlier, in this case we do it simultaneously before and after the performance of a {\it bbm}.

\begin{prop}{\cite[Proposition~4]{DLP}}
Equations of the infinite system obtained from elements in $\Lambda^{aug}$ followed by braiding tails in ${\rm H}_{n}(q)$ are equivalent to equations obtained from 
elements in $\Lambda$ followed by braiding tails, where a braid band move can be performed on any moving strand.
\end{prop}

\begin{proof}
It follows from Theorem~9, since all steps followed so as to order the exponents in a monomial in the $t_i$'s remain
the same after the performance of a {\it bbm}, ignoring the $t^p$ appearing after the {\it bbm}.
\end{proof}

\subsection{From the ${\rm H}_{n}(q)$-module $\Lambda$ to $\Lambda$: Eliminating the tails}

We now deal with the `braiding tails'. Applying the same technique as in Theorem~10 before and after the performance of a {\it bbm}, we first prove that equations obtained by performing {\it bbm}'s on any moving strand on elements in $\Lambda$ followed by words in ${\rm H}_n(q)$, reduce to equations obtained by performing {\it bbm}'s on any moving strand from elements in $\Lambda^{aug}$ (with no `braiding tails').

\smallbreak

\begin{prop}{\cite[Proposition~5]{DLP}}\label{tailbbm}
Equations obtained from {\it bbm}'s on elements in $\Lambda$ followed by words in ${\rm H}_n(q)$ are equivalent to equations obtained by performing a braid band move on any moving strand on elements in $\Lambda^{aug}$ .
\end{prop}

\begin{proof}
We perform a {\it bbm} on an element $a\cdot w$ in the ${\rm H}_n(q)$-module $\Lambda$ and we cable the parallel strand with the surgery strand, see Figure~\ref{tailfig}. We then apply Theorem~10 before and after
the performance of the {\it bbm} and uncable the parallel strand. The proof is illustrated in Figure~\ref{tailfig}.
\end{proof}

\begin{figure}[!ht]
\begin{center}
\includegraphics[width=5.4in]{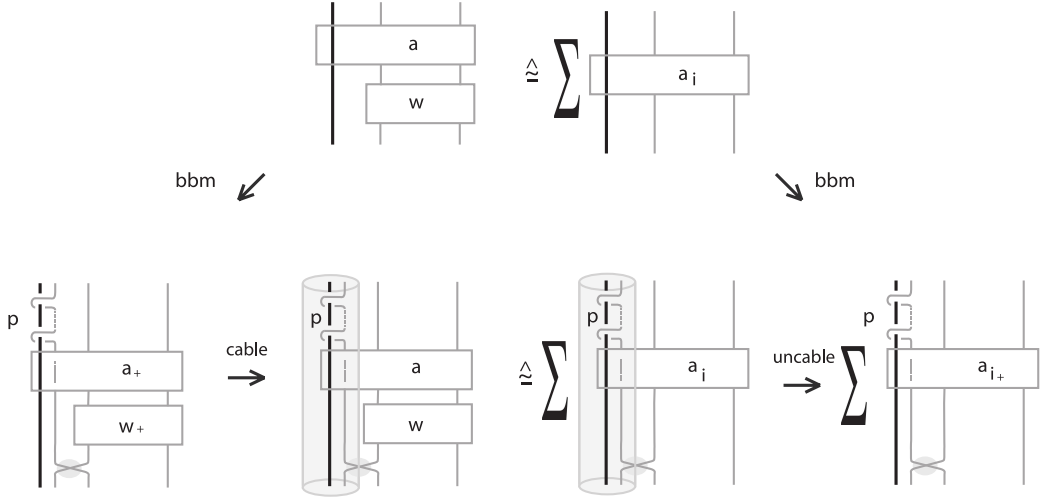}
\end{center}
\caption{ The proof of Proposition~\ref{tailbbm}. }
\label{tailfig}
\end{figure}

\begin{ex}\rm
In this example we demonstrate Proposition~\ref{tailbbm}.
\[
\begin{matrix}
tt_1t_2\cdot g_1g_2g_1 & \underset{bbm}{\overset{1^{st} str.}{\longrightarrow}} & t^p t_1t_2t_3 \cdot g_2g_3g_2 g_1^{\pm 1} \\
\mid & & \mid \\
elim.\ tails     & & elim.\ tails\\
\downarrow & & \downarrow \\
(q-1)(q^2-q+1)\cdot tt_1t_2 & \underset{bbm}{\overset{1^{st} str.}{\longrightarrow}} & (q-1)(q^2-q+1)\cdot t^pt_1t_2t_3 g_1^{\pm1} \\
+ & & + \\
q(q-1)^2z \cdot tt_1^2 & \underset{bbm}{\overset{1^{st} str.}{\longrightarrow}} & q(q-1)^2z \cdot t^pt_1t_2^2 g_1^{\pm1}\\
+ & & + \\
\left[q^2(q-1)(q^2-q+1)z^2\right]\cdot t^3 & \underset{bbm}{\overset{1^{st} str.}{\longrightarrow}} & \left[q^2(q-1)(q^2-q+1)z^2\right]\cdot t^pt_1^3g_1^{\pm1}\\
+ & & + \\
a\cdot t^2t_1 & \underset{bbm}{\overset{1^{st} str.}{\longrightarrow}} & a\cdot t^pt_1^2t_2g_1^{\pm1}\\
\end{matrix}
\]
\smallbreak
\noindent where $a= q^3z+q^2(q-1)^2+2q^2(q-1)^2z+q(q-1)^4z$.
\end{ex}

In the process of eliminating the `braiding tails' we have so far reached {\it bbm}'s in the set $\Lambda^{aug}$ (with no `braiding tails'). Following the same procedure and applying the same techniques as in \cite{DL2}, we finally reach the set $\Lambda$ by means of the following:

\begin{thm}{\cite[Theorem~8]{DLP}}
In order to obtain the {\it bbm} equations needed for computing the HOMFLYPT skein module of the lens spaces $L(p,1)$, it suffices to perform braid band moves on any strand on elements in the basis $\Lambda$ of $\mathcal{S}({\rm ST})$.
\end{thm}

\begin{proof}
The proof is based on Theorems~8, 9 \& 10 and the fact that the {\it bbm}'s commute with the stabilization moves (Lemma~1)
and the skein (quadratic) relation (Lemma~2).
\end{proof}

\begin{remark}\rm
The fact that the {\it bbm}'s and conjugation do not commute, results in the need of performing braid band moves on all moving strands of elements in $\Lambda$.
\end{remark}

\begin{figure}
\begin{center}
\includegraphics[width=2.5in]{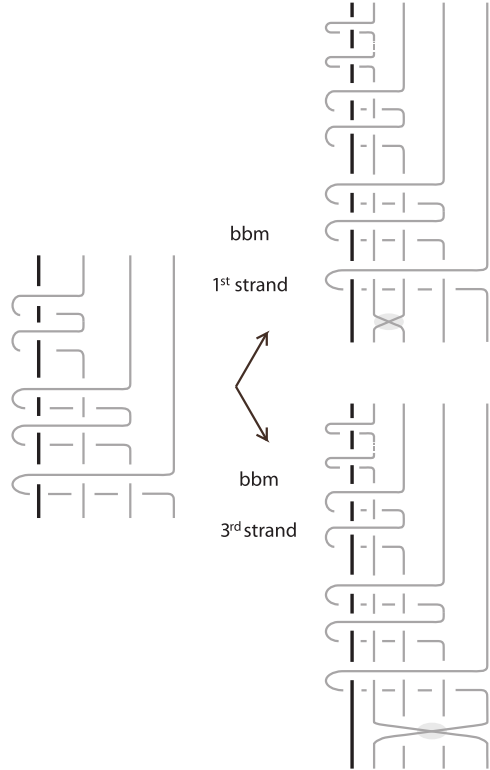}
\end{center}
\caption{The performance of a {\it bbm} on the $1^{st}$ and on the $3^{rd}$ strand of an element in $\Lambda$.}
\label{exlambda}
\end{figure}

\section{Conclusions}

In this paper we present recent results toward the computation of the HOMFLYPT skein module of the lens spaces $L(p,1)$, $\mathcal{S}\left(L(p,1)\right)$, via braids. We first presented a new basis $\Lambda$ of $\mathcal{S}({\rm ST})$ in braid form, which is crucial for the braid approach to $\mathcal{S}\left(L(p,1)\right)$ and we then related $\mathcal{S}\left(L(p,1)\right)$ to $\mathcal{S}({\rm ST})$ by means of equations resulting from {\it bbm}'s. In particular, we showed that by considering elements in the basis $\Lambda$ and imposing on values on them of the Lambropoulou invariant $X$ for knots and links in ST relations coming from the performance of {\it bbm}'s on all their moving strands, we arrive at an infinite system of equations, the solution of which coincides with the computation of $\mathcal{S}\left(L(p,1)\right)$. Our results are summarized to the following: 

\bigbreak

{\it
\noindent In order to compute $\mathcal{S}(L(p,1))$ it suffices to solve the infinite system of equations:

$$ X_{\widehat{\tau}}\ =\ X_{\widehat{bbm_i(\tau)}},$$

\noindent where $bbm_i(\tau)$ is the result of the performance of {\it bbm} on the $i^{th}$-moving strand of $\tau \in \Lambda$, for all $\tau \in \Lambda$ and for all $i$.
}

\bigbreak

In \cite{DL3} we elaborate on the system. This is a very technical and difficult task. These results will serve as our main tool for computing $\mathcal{S}\left(L(p,q)\right)$ in general.

\end{document}